\documentclass[a4paper, 11pt]{article}
\title{Equivariant zeta functions for invariant Nash germs}
\author{Fabien Priziac}
\date{}

\usepackage[pdf]{pstricks}

\usepackage{amsfonts}
\usepackage{amsmath}
\usepackage{amsthm}
\usepackage{amssymb}
\usepackage[all]{xy}
\usepackage{graphicx}

\newtheorem{de}{Definition}[section]
\newtheorem{theo}[de]{Theorem}
\newtheorem{prop}[de]{Proposition}

\newtheorem{lem}[de]{Lemma}

\newcommand{\Hom}{\text{Hom}}

\theoremstyle{remark}
\newtheorem{rem}[de]{Remark}
\newtheorem{ex}[de]{Example}

\begin{document}

\maketitle

\begin{abstract} To any Nash germ invariant under right composition with a linear action of a finite group, we associate its equivariant zeta functions, inspired from motivic zeta functions, using the equivariant virtual Poincar\'e series as a motivic measure. We show Denef-Loeser formulae for the equivariant zeta functions and prove that they are invariants for equivariant blow-Nash equivalence via equivariant blow-Nash isomorphisms. Equivariant blow-Nash equivalence between invariant Nash germs is defined as a generalization involving equivariant data of the blow-Nash equivalence.  
\end{abstract}
\footnote{
Keywords : Nash germs, group action, blow-Nash equivalence, zeta functions, equivariant virtual Poincar\'e series, Denef-Loeser formula.
\\
{\it 2010 Mathematics Subject Classification :} 14B05, 14P20, 14P25, 32S15, 57S17, 57S25.
 
Research supported by a Japan Society for the Promotion of Science (JSPS) Postdoctoral Fellowship.
}

\section{Introduction}

A crucial issue in the study of real analytic germs is the choice of a good equivalence relation by which we can distinguish them. One may think about $C^r$-equivalence, $r=0,1,\ldots,\infty, \omega$. However, the topological equivalence seems, unlike the complex case, not fine enough : for example, all the germs of the form $x^{2m} + y^{2n}$ are topologically equivalent. On the other hand, the $C^1$-equivalence has already moduli : consider the Whitney family $f_t(x,y) = xy(y-x)(y-tx)$, $t >1$, then $f_t$ and $f_{t'}$ are $C^1$-equivalent if and only if $t~=~t'$. In \cite{Kuo}, T.-C. Kuo proposed an equivalence relation for real analytic germs named the blow-analytic equivalence for which, in particular, analytically parametrized family of isolated singularities have a locally finite classification. Roughly speaking, two real analytic germs are said blow-analytically equivalent if they become analytically equivalent after composition with real modifications (e.g., finite successions of blowings-up along smooth centers). With respect to this equivalence relation, Whitney family has only one equivalence class. Slightly stronger versions of blow-analytic equivalence have been proposed so far, by S. Koike and A. Parusi\'nski in \cite{KoiPar} and T. Fukui and L. Paunescu in \cite{FukuiPau} for example. An important feature of blow-analytic equivalence is also that we have invariants for this equivalence relation, like the Fukui invariants (\cite{Fukui}) and the zeta functions (\cite{KoiPar}) inspired by the motivic zeta functions of J. Denef and F. Loeser (\cite{DL-GA}) using the Euler characteristic with compact supports as a motivic measure.

The present paper is interested in the study of Nash germs, that is real analytic germs with semialgebraic graph. In \cite{GF-ZF}, G. Fichou defined an analog adapted to Nash germs of the blow-analytic equivalence of T.-C. Kuo in \cite{Kuo} : two Nash germs are said blow-Nash equivalent if, after composition with Nash modifications, they become analytically equivalent via a Nash isomorphism (if the Nash isomorphism preserves the critical loci of the Nash modifications, it is called a blow-Nash isomorphism). He showed in particular that blow-Nash equivalence is an equivalence relation and that it has no moduli for Nash families with isolated singularities. Using as a motivic measure the virtual Poincar\'e polynomial of C. McCrory and A. Parusi\'nski in \cite{MCP-VB}, extended to the wider category of $\mathcal{AS}$ sets (\cite{Kur} and \cite{KP}) by G. Fichou in \cite{GF-MI}, one can generalize the zeta functions of S. Koike and A. Parusi\'nski in \cite{KoiPar}. In \cite{GF-ZF}, G. Fichou showed that these latter zeta functions are invariants for blow-Nash equivalence via blow-Nash isomorphisms. 

In this paper, we consider Nash germs invariant under right composition with a linear action of a finite group. We define for such germs a generalization of the blow-Nash equivalence of $\cite{GF-ZF}$ involving equivariant data. If $G$ is a finite group acting linearly on $\mathbb{R}^d$ and trivially on $\mathbb{R}$, we say that two equivariant, or invariant, Nash germs $f, h : (\mathbb{R}^d,0) \rightarrow (\mathbb{R},0)$ are $G$-blow-Nash equivalent if there exist two equivariant Nash modifications $\sigma_f : (M_f, \sigma_f^{-1}(0)) \rightarrow (\mathbb{R}^d,0)$ and $\sigma_h : (M_h, \sigma_h^{-1}(0)) \rightarrow (\mathbb{R}^d,0)$ of $f$ and $h$ and an equivariant Nash isomorphism $\Phi : (M_f, \sigma_f^{-1}(0)) \rightarrow (M_h, \sigma_h^{-1}(0))$ which induces an equivariant homeomorphism $\phi : (\mathbb{R}^d,0) \rightarrow (\mathbb{R}^d,0)$ such that $f = h \circ \phi$ (Definition \ref{defequivbln}). If $\Phi$ preserves the critical loci of $\sigma_f$ and $\sigma_h$, we say that $\Phi$ is an equivariant blow-Nash isomorphism. We consider the equivalence relation generated by the equivariant blow-Nash equivalence, which allows refinement of the non-equivariant blow-Nash classification. For example, consider the germs $y^4 - x^2$ and $x^4 - y^2$. They are Nash equivalent but we show in Example \ref{exgermnonequivbln} that they are not $G$-blow-Nash equivalent via an equivariant blow-Nash isomorphism if $G = \{1, s\}$ with $s$ the involution given by $(x,y) \mapsto (-x,y)$.

Our main interest is the construction of invariants for $G$-blow-Nash equivalence via equivariant blow-Nash isomorphism. We associate to any invariant Nash germ its equivariant zeta functions : they are defined using the equivariant virtual Poincar\'e series of G. Fichou in \cite{GF} as an equivariant motivic measure on its arc spaces equipped with the induced action of $G$ (section \ref{subsecequivzetafunc}). It is a generalization of the zeta functions defined in \cite{GF-MI} and \cite{GF-ZF}, and they are different from the equivariant zeta functions defined in \cite{GF}. We then prove the rationality of the equivariant zeta functions by Denef-Loeser formulae (Propositions \ref{equivDLform} and \ref{equivDLformsigns}). One has to keep attention on the behaviour of the induced actions of $G$ on all the spaces involved in the demonstrations of the formulae. A key point is the proof of the validity of Kontsevich ``change of variables formula'' (\cite{Kon}) in this equivariant setting (Proposition \ref{Konequiv}).   

Finally, we compute the equivariant zeta functions of several invariant Nash germs (section \ref{examples}). We are in particular interested in the invariant Nash germs induced from the normal forms of the simple boundary singularities of manifolds with boundary (see \cite{AGZV}). In a subsequent work, we plan to study the simple boundary singularities of Nash manifolds with boundary and classify them with respect to equivariant blow-Nash equivalence.
\\

We begin this paper by the definition of $G$-blow-Nash equivalence for $G$ a finite group. We make also precise what we mean by an equivariant modification of an invariant Nash germ.

In section \ref{sectionequivzetafunc}, we define the equivariant zeta functions (naive and with signs) of an invariant Nash germ. We first recall the definition of the $G$-equivariant virtual Betti numbers~: they are the unique additive invariants on the category of $\mathcal{AS}$ sets equipped with an algebraic action of $G$ which coincide with the dimensions of equivariant Borel-Moore homology with $\mathbb{Z}_2$-coefficients on compact nonsingular sets. In subsection \ref{parequivDLform}, we prove an equivariant version of Kontsevich ``change of variables formula'' and Denef-Loeser formulae for equivariant zeta functions.

In section \ref{secezfebne}, we show that the equivariant zeta functions are invariant under equivariant blow-Nash equivalence via equivariant blow-Nash isomorphisms, illustrating this result with the example of the Nash germs $y^4 - x^2$ and $x^4 - y^2$ invariant under the involution $(x,y) \mapsto (-x,y)$. The computation of the equivariant zeta functions of several other invariant Nash germs concludes the paper. 
\\
  
{\bf Acknowledgements.} The author wishes to thank G. Fichou and T. Fukui for useful discussions and comments.

\section{Equivariant blow-Nash equivalence} \label{equivbnequiv}

Let $G$ be a finite group.

We are interested in the study of germs of Nash functions invariant under some linear action of $G$ on the source space. More precisely, we want to make progress towards the classification of such germs up to equivariant equivalence. We define below in \ref{defequivbln} some generalization of the blow-Nash equivalence defined by G. Fichou in \cite{GF-ZF}, taking into account the equivariant data of this setting. 

Let us first make precise definitions in the equivariant setting. Let $d \geq 1$ and equip the affine space $\mathbb{R}^d$ with a linear action of $G$ and the real line $\mathbb{R}$ with the trivial action of $G$. In this setting, a germ of an equivariant Nash function $f : (\mathbb{R}^d,0) \rightarrow (\mathbb{R}, 0)$ will be called an equivariant or invariant Nash germ.

An equivariant Nash modification of such a germ $f$ will be an equivariant proper surjective Nash map $\pi : (M, \pi^{-1}(0)) \rightarrow (\mathbb{R}^d,0)$ between $G$-globally stabilized semialgebraic and analytic neighbourhoods of $\pi^{-1}(0)$ in $M$ and $0$ in $\mathbb{R}^d$, such that
\begin{enumerate}
	\item $M$ is a Nash manifold equipped with an algebraic action of $G$ (that is an action induced from a regular $G$-action on the Zariski closure of $M$), given by algebraic isomorphisms $\delta_g$, $g \in G$,
	\item the equivariant complexification $\pi(\mathbb{C}) : M(\mathbb{C}) \rightarrow \mathbb{C}^d$ is an equivariant biholomorphism outside some subset of $M(\mathbb{C})$ of codimension at least $1$, globally stabilized by the complexified action of $G$ on $M(\mathbb{C})$,
	\item $\pi$ is an isomorphism outside the zero locus of $f$, 
	\item the irreducible components of $(f \circ \pi)^{-1}(0)$ which are not exceptional divisors of $\pi$ do not intersect,  
	\item the action of $G$ on $M$ preserves globally each exceptional divisor of $\pi$,
%	\item the action of $G$ on $M$ preserves the multiplicities of the jacobian determinant of $\pi$ along its exceptional divisors,
	\item the composition $f \circ \pi$ and the jacobian determinant $jac~\pi$ of $\pi$ have only normal crossings simultaneously, on which the action of $G$ on $M$ can be locally linearized in the following meaning :
\\

	Let $(f \circ \pi)^{-1}(0) = \bigcup_{j \in J} E_j$ be the decomposition of $(f \circ \pi)^{-1}(0)$ into irreducible components. For $I \subset J$, we denote $E_I := \bigcap_{i \in I} E_i$. We ask that for any $I \subset J$ with $|I| \leq d$, for any element $x$ of $E_I$, there exists an affine open neighborhood $U_x$ of $x$ in $M$, an affine open neighborhood $V_x$ of $0$ in $\mathbb{R}^d$, with coordinates $y_1, \ldots, y_d$ and a Nash isomorphism $\varphi_x : V_x \rightarrow U_x$ (in the sense of \cite{GF-MI}) such that
	\begin{enumerate}
		\item for all $i \in I$, there exists $j_i \in \{1, \cdots, d\}$, such that
			\begin{itemize} 
				\item $E_i \cap U_{x} = \varphi_{x}(\{y_{j_i} = 0\} \cap V_{x})$,
				\item $f \circ \pi(\varphi_x(y_1, \ldots, y_d)) = unit(y_1, \ldots, y_d)~\prod_{i \in I} y_{j_i}^{N_{j_i}}$,
				\item $jac~\pi(\varphi_x(y_1, \ldots, y_d)) = unit(y_1, \ldots, y_d)~\prod_{i \in I} y_{j_i}^{\nu_{j_i} - 1}$,
			\end{itemize}
		\item for all $g \in G$, $\delta_g(E_i) \cap \delta_g(U_{x}) = \varphi_{g \cdot x}(\{y_{j_i} = 0\} \cap V_{g \cdot x})$,
		\item for all $g \in G$, $\delta_g(U_x) = U_{g \cdot x}$ and there exists a linear isomorphism $\nu_{x,g} : \mathbb{R}^d \rightarrow \mathbb{R}^d$ such that $\nu_{x,g}(V_x) = V_{g \cdot x}$ making the following diagram commute :
		$$\xymatrix{
V_x \ar[rr]^{\varphi_x} \ar[d]_{\nu_{x,g}} && U_x \ar[d]^{\delta_g} \\
V_{g \cdot x} \ar[rr]^{\varphi_{g \cdot x}} && U_{g\cdot x}
}
$$ 
		\item if $\delta_g(E_I) = E_I$, $U_{g \cdot x} = U_x$, $V_{g \cdot x} = V_x$ and $\varphi_{g \cdot x} = \varphi_x$, 
%		\item if $I$ contains (exactly) an element $k$ of $J \setminus K$, for all $g \in G$, $\delta_g(E_k) \cap U_{g \cdot x} = \varphi_{g \cdot x} (\{ y_{j_k} = 0 \} \cap V_{g \cdot x})$ and $E_k \cap U_{g \cdot x} = \varphi_{g \cdot x} (\{ y_{j_i} = 0 \} \cap V_{g \cdot x})$,
%		\item for all $g_1$, $g_2$ of $G$, $\nu_{g_2 \cdot x,g_1} \circ \nu_{x,g_2} = \nu_{x, g_1 g_2}$,
		\item for all $g \in G$, $\nu_{x,g}$ preserves the intersection of the hyperplanes $\{y_s = 0\}$, $s \notin \{j_i, i \in I\}$, 
%onto the intersection of the hyperplanes $\{y_r = 0\}$, $r \notin \{j_{g \cdot i}, i \in I\}$,
		\item for all $g \in G$, the linear isomorphisms $\nu_{h \cdot x,g}$, $h \in G$, are all given by the same matrix $A_{x,g}$ in the canonical bases of $\mathbb{R}^d \supset V_{h \cdot x}$ and $\mathbb{R}^d \supset V_{gh \cdot x}$,
		\item all these conditions come from the semialgebraic and analytic isomorphisms between compact semialgebraic and real analytic sets inducing the Nash isomorphisms $\varphi_x$.
	\end{enumerate} 
\end{enumerate}

	%The action of $G$ on $M$ induces an action of $G$ on the set of irreducible components of $(f \circ \sigma)^{-1}(0)$. For $j_1, j_2 \in J$ and $g \in G$, we write the equality $j_2 = g \cdot j_1$ if $E_{j_2} = g \cdot E_{j_1}$. 
	%This induces an action of $G$ on the set $\Lambda$ of non-empty subsets of $J$ and we denote by $\underline{I}$ the orbit of a non-empty subset $I$ of $J$. For $\underline{I}$ in $\Lambda/G$, we then denote by $E_{\underline{I}}$ the union of the sets $E_{g \cdot I} = g \cdot E_I = \bigcap_{i \in I} g \cdot E_i$, $g \in G$ (it is the orbit of $E_I$ in $M$). 

\begin{de} \label{defequivbln} Let $f,h : (\mathbb{R}^d,0) \rightarrow (\mathbb{R}, 0)$ be two invariant Nash germs. We say that $f$ and $h$ are $G$-blow-Nash equivalent if there exist
\begin{itemize}
	\item two equivariant Nash modifications $\sigma_f : (M_f, \sigma_f^{-1}(0)) \rightarrow (\mathbb{R}^d,0)$ and $\sigma_h : (M_h, \sigma_h^{-1}(0)) \rightarrow (\mathbb{R}^d,0)$ of $f$ and $h$ respectively,
	\item an equivariant Nash isomorphism $\Phi$ between $G$-globally stabilized semialgebraic and analytic neighbourhoods $(M_f, \sigma_f^{-1}(0))$ and $(M_h, \sigma_h^{-1}(0))$,
	\item an equivariant homeomorphism $\phi : (\mathbb{R}^d,0) \rightarrow (\mathbb{R}^d,0)$,
\end{itemize}
such that the following diagram commutes :
$$\xymatrix{
(M_f, \sigma_f^{-1}(0)) \ar[rr]^{\Phi} \ar[d]_{\sigma_f} && (M_h, \sigma_h^{-1}(0)) \ar[d]^{\sigma_h} \\
(\mathbb{R}^d,0) \ar[rr]^{\phi} \ar[rd]^f && (\mathbb{R}^d,0) \ar[ld]_h \\
&(\mathbb{R},0)
}
$$

In this case, we say that $\phi$ is an equivariant blow-Nash homeomorphism, and if $\Phi$ preserves the multiplicities of the jacobian determinant of $\sigma_f$ and $\sigma_g$ along their exceptional divisors, then we say that $\Phi$ is an equivariant blow-Nash isomorphism.
\end{de}

\begin{rem} 
\begin{itemize}
	\item If $G = \{e\}$, the equivariant blow-Nash equivalence is the blow-Nash equivalence defined in \cite{GF-ZF}.
	\item There exist germs being blow-Nash equivalent via a blow-Nash isomorphism without being $G$-blow Nash equivalent via an equivariant blow-Nash isomorphism : see Example \ref{exgermnonequivbln}.	
\end{itemize}
\end{rem}

	In the following, we will also call $G$-blow-Nash equivalence (resp. $G$-blow-Nash equivalence via an equivariant blow-Nash isomorphism) the equivalence relation generated by the $G$-blow-Nash equivalence (resp. $G$-blow-Nash equivalence via an equivariant blow-Nash isomorphism) defined in Definition \ref{defequivbln}. Notice that the $G$-blow-Nash equivalence can be defined if $G$ is an infinite group as well.

%\begin{prop}
%The equivariant blow-Nash equivalence is an equivalence relation for invariant Nash germs.
%\end{prop}

%\begin{proof}
%The proof runs as in \cite{GF-ZF}, involving in our framework equivariant data, since we can equip the fiber product of equivariant maps with the induced diagonal action, making projection maps equivariant. Notice that the action induced by $G$ on the fiber products preserves the jacobian determinants of the associated projections. 
%\end{proof}

%\begin{rem}
%The $G$-blow-Nash equivalence can also be defined if $G$ is an infinite group, and it is still an equivalence relation in this case.
%\end{rem}

\section{Equivariant zeta functions} \label{sectionequivzetafunc}

Let $G$ be a finite group.
\\

We are interested in the classification of Nash germs invariant under right composition with a linear action of $G$, with respect to the equivariant blow-Nash equivalence. With this in mind, we generalize the zeta functions defined in \cite{GF-MI} to our equivariant setting, using the equivariant virtual Poincar\'e series defined in \cite{GF}. We show in Proposition \ref{equivDLform} the rationality of our equivariant zeta functions by a Denef-Loeser formula, which allows us to prove that they are invariants for equivariant blow-Nash equivalence via an equivariant blow-Nash isomorphism (Theorem \ref{equivzetafuncinv}).

\subsection{Equivariant virtual Poincar\'e series} 

In order to define ``equivariant'' generalizations of the zeta functions for Nash germs, we use an additive invariant defined on all $G$-$\mathcal{AS}$ sets, that is boolean combinations of arc-symmetric sets (see \cite{Kur} and \cite{KP}) equipped with an algebraic action of $G$ : the equivariant virtual Poincar\'e series. It is defined in \cite{GF} using the equivariant virtual Betti numbers, which are the unique additive invariant on $G$-$\mathcal{AS}$ sets coinciding with the dimensions of their equivariant homology. In this subsection, we recall the results of G. Fichou in \cite{GF} about equivariant Betti numbers. We first give the definition of equivariant homology which is a mix of group cohomology and Borel-Moore homology.

\begin{de} Let $\mathbb{Z}_2[G]$ denote the group ring of $G$ over $\mathbb{Z}_2$, that is 
$$\mathbb{Z}_2[G]=\left\{\sum_{g \in G} n_g g~|~g \in G \right\}$$
equipped with the induced ring structure. Consider a projective resolution $(F_*, \Delta_*)$ of $\mathbb{Z}_2$ by $\mathbb{Z}_2[G]$-modules, that is vector spaces over $\mathbb{Z}_2$ equipped with a linear action of $G$. Then we define the cohomology $H^*(G,M)$ of the group $G$ with coefficients in a $\mathbb{Z}_2[G]$-module $M$ to be the cohomology of the cochain complex
$$\left(\rm{Hom}_{\mathbb{Z}_2[G]}(F_*, M), \Delta^*\right)$$  
where, if $\varphi : F_k \rightarrow M$ is an equivariant linear morphism, $\Delta^k(\varphi) := \varphi \circ \Delta_{k+1}$.
\end{de}

\begin{ex} Let $G$ be a finite cyclic group of order $d$ generated by $s$. We denote by $N := \sum_{1 \leq i \leq d} s^i$. Then a projective resolution of $\mathbb{Z}_2$ by $\mathbb{Z}_2[G]$-modules is given by
$$\cdots \rightarrow \mathbb{Z}_2[G] \xrightarrow{1+s} \mathbb{Z}_2[G] \xrightarrow{N} \mathbb{Z}_2[G] \xrightarrow{1+s} \mathbb{Z}_2[G] \rightarrow  \mathbb{Z}_2 \rightarrow 0,$$
where the map $ \mathbb{Z}_2[G] \rightarrow  \mathbb{Z}_2$ associates to an element $\sum_{1 \leq i \leq d} n_i s^i$ of $\mathbb{Z}_2[G]$ the element $\sum_{1 \leq i \leq d} n_i$ of $\mathbb{Z}_2$.

The cohomology of the group $G$ with coefficients in a $\mathbb{Z}_2[G]$-module $M$ is
$$H^n(G,M) =
\begin{cases}
\frac{M^G}{N M} & \mbox{   if $n$ is an even positive integer,} \\
\frac{ker~N}{(1+s) M } & \mbox{   if $n$ is an odd positive integer,} \\ 
M^G & \mbox{   if $n =0 $}
\end{cases}
$$
(where $M^G$ denotes the set of elements of $M$ which are fixed by the action of $G$). In particular, if $G = \mathbb{Z}/2 \mathbb{Z}$, 
$$H^n(G,M) = 
\begin{cases}
\frac{M^G}{(1 + s)M} & \mbox{   if $n > 0$,} \\ 
M^G  & \mbox{   if $n =0 $.}
\end{cases}
$$
\end{ex} 
For more details about group cohomology see for instance \cite{Bro}.
\\

The equivariant homology of $G$-$\mathcal{AS}$ sets we define below is inspired by \cite{VH}. 

Recall that a semialgebraic subset $S$ of $\mathbb{P}^n(\mathbb{R})$ is said to be arc-symmetric if every real analytic arc in $\mathbb{P}^n(\mathbb{R})$ either meets $S$ at isolated points or is entirely included in $S$. An $\mathcal{AS}$ set is a boolean combination of arc-symmetric sets.

Take $X$ an $\mathcal{AS}$ set equipped with an algebraic action of $G$, that is an action induced from a regular $G$-action on its Zariski closure : we will call such a set a $G$-$\mathcal{AS}$ set. We can associate to $X$ the complex $(C_*(X), \partial_*)$ of its semialgebraic chains with closed supports and $\mathbb{Z}_2$ coefficients, which computes the Borel-Moore homology of $X$ with $\mathbb{Z}_2$ coefficients, simply denoted by $H_*(X)$~: see Appendix of \cite{MCP}. The action of $G$ on $X$ induces by functoriality a $G$-action on the chain complex $C_*(X)$ (linear action on chains in each dimension and commutativity with the differential). We then consider the double complex
$$(\Hom_{\mathbb{Z}_2[G]}(F_{-p}, C_q(X)))_{p,q \in \mathbb{Z}},$$ 
where $(F_*, \Delta_*)$ is a projective resolution of $\mathbb{Z}_2$ by $\mathbb{Z}_2[G]$-modules, where the differentials are induced by $\Delta_*$ and $\partial_*$.
\\

The equivariant Borel-Moore homology $H_*(X ; G)$ of $X$ (with $\mathbb{Z}_2$ coefficients) is then by definition the homology of the total complex associated to the above double complex. 
\\

Such a double complex induces two spectral sequences that converge to the homology of the associated total complex. In particular, the spectral sequence given by
$$E^2_{p,q} = H^{-p}(G, H_q(X)) \Rightarrow H_{p+q}(X ; G),$$
is called the Hochschild-Serre spectral sequence of $X$ and $G$.
 
It gives the following wiewpoint on the equivariant Borel-Moore homology : it is a mix of group cohomology and Borel-Moore homology with $\mathbb{Z}_2$ coefficients, involving the geometry of $X$, the geometry of the action of $G$ and the geometry of the group $G$ itself.

\begin{ex} \label{exsphere} To illustrate how the equivariant geometry is involved in the equivariant homology, let us compute the equivariant homology of the two-dimensional sphere, given by the equation $x^2 + y^2 + z^2 = 1$ in $\mathbb{R}^3$ and denoted by $X$, equipped with two different kind of involutions.

Consider first the action given by the central symmetry $s : (x,y,z) \mapsto (-x,-y,-z)$. If $G := \{1, s\}$, the $E^2$-term of the Hochschild-Serre spectral sequence of $X$ and $G$ is 

$$\xymatrix{
\cdots & \mathbb{Z}_2[X] & \mathbb{Z}_2[X] & \mathbb{Z}_2[X]  &  \mathbb{Z}_2[X]  \\
\cdots & 0 & 0 &  0 \ar[llu] &  0 \ar[llu] \\
\cdots & \mathbb{Z}_2\overline{[p]} & \mathbb{Z}_2\overline{[p]} &  \mathbb{Z}_2\overline{[p]} \ar[llu] &  \mathbb{Z}_2\overline{[p]} \ar[llu]  
}$$
where $\overline{[p]}$ is the homology class of the chain $[p]$ representing a point $p$ of $X$ : for the sake of simplicity in the computations, we choose $p$ to be the point of coordinates $(1,0,0)$. We see that the differential $d^2$ vanishes everywhere and $E^3$-term is then given by 
$$\xymatrix{
\cdots & \mathbb{Z}_2[X] & \mathbb{Z}_2[X] & \mathbb{Z}_2[X]  &  \mathbb{Z}_2[X]  \\
\cdots & 0 & 0 &  0 &  0  \\
\cdots & \mathbb{Z}_2\overline{[p]} & \mathbb{Z}_2\overline{[p]} &  \mathbb{Z}_2\overline{[p]} \ar[llluu] &  \mathbb{Z}_2\overline{[p]} \ar[llluu]  
}$$
The image of $\overline{[p]}$ by the differential $d^3$ can be obtained by the following procedure. We follow the following ``path'' in the double complex $(\Hom_{\mathbb{Z}_2[G]}(F_{-p}, C_q(X)))_{p,q \in \mathbb{Z}}$ :

$$\xymatrix{
C_2(X)& C_2(X)  \ar[l]^{1+s}  \ar[d]^{\partial_2} & C_2(X)  &  C_2(X) \\
C_1(X) & C_1(X)  &  C_1(X) \ar[l]^{1+s} \ar[d]^{\partial_1} &  C_1(X) \\
C_0(X) & C_0(X) &  C_0(X) &  C_0(X) \ar[l]^{1+s} 
}$$

Apply $1+s$ to the chain $[p]$. There exists a semialgebraic chain $\gamma$ of $C_1(X)$ with boundary $[p] + s([p]) = [\{p, s(p)\}]$ : we can choose $\gamma$ to be the chain representing an arc of the equator $\{z=0\}$ of $X$. The image of $\gamma$ by $1+ s$ is the chain representing the whole equator $\{z=0\}$, which is the semialgebraic boundary of the half-sphere $\{z \geq 0\}$. Finally, if we apply $1+ s$ to the chain representing the half-sphere, we obtain the chain $[X]$ representing the whole sphere. Therefore $d^3(\overline{[p]}) = [X]$. 

Consequently, 
$$E^{\infty}_{p,q} = E^4_{p,q} = \begin{cases} \mathbb{Z}_2[X] \mbox{ if $q = 2$ and $-2 \leq p \leq 0$,} \\
0 \mbox{ otherwise,}
\end{cases}$$
and $H_n(X ; G) = \begin{cases} \mathbb{Z}_2 \mbox{ if $0 \leq n \leq 2$,} \\
0 \mbox{ otherwise.}
\end{cases}$
\\

Now let $s$ denote an involution on $X$ which is not free : this means there exists at least one point $p_0$ of $X$ that is fixed by $s$. If we look at the $E^3$-term of the Hochschild-Serre spectral sequence of $X$ with respect to this action of $G = \mathbb{Z}/2\mathbb{Z}$, we see that the differential $d^3$ vanishes everywhere since $H_0(X) = \mathbb{Z}_2 \overline{[p_0]}$ and $(1+ s)[p_0] =0$. Thus, $E^{\infty} = E^2$ and
$$H_n(X ; G) = \begin{cases} \mathbb{Z}_2 \mbox{ if $0 \leq n \leq 2$,} \\
\mathbb{Z}_2 \oplus \mathbb{Z}_2 \mbox{ if $n \leq 0$,} \\
0 \mbox{ otherwise.}
\end{cases}$$  
\end{ex}

\begin{rem}
\begin{itemize}
	\item When $G = \{e\}$, the equivariant homology of a $G$-$\mathcal{AS}$ set $X$ is the Borel-Moore homology of $X$.
	\item As illustrated in Example \ref{exsphere}, the equivariant homology groups can be non-zero in negative degree. In the case $G = \mathbb{Z}/2\mathbb{Z}$, we actually have $H_n(X ; G) \cong \oplus_{i \geq 0} H_i(X^G)$ for $n < 0$ (where $X^G$ is the set of the points of $X$ which are fixed by the action of $G$).
\end{itemize}
\end{rem}

For more details about equivariant Borel-Moore homology, see \cite{VH}, \cite{Der}, \cite{GF} and \cite{Pri-EWF}.
\\

The existence and uniqueness of the equivariant Betti numbers are given by the following theorem of G. Fichou in \cite{GF}. The equivariant virtual Betti numbers and the equivariant virtual Poincar\'e series are additive invariants under equivariant Nash isomorphisms of $G$-$\mathcal{AS}$ sets. By a Nash isomorphism between $\mathcal{AS}$-sets $X_1$ and $X_2$ is meant the restriction of a semialgebraic and analytic isomorphism between compact real analytic and semialgebraic sets $Y_1$ and $Y_2$ containing $X_1$ and $Y_2$ respectively (see also \cite{GF-MI}). 

We will use the equivariant virtual Poincar\'e series as a measure for arc spaces which takes into account equivariant information. In particular, we will apply it to the spaces of arcs of an invariant Nash germ and gather these measures in the equivariant zeta functions (subsection \ref{subsecequivzetafunc}).

\begin{theo}[Theorem 3.9 of \cite{GF}] \label{eqvirpoiser} Let $i \in \mathbb{Z}$. There exists a unique map $\beta^G_i(\cdot)$ defined on $G$-$\mathcal{AS}$ sets and with values in $\mathbb{Z}$ such that
\begin{enumerate}
	\item $\beta_i^G(X_1) = \beta_i^G(X_2)$ if $X_1$ and $X_2$ are equivariantly Nash isomorphic,
  \item $\beta_i^G(X) = \dim_{\mathbb{Z}_2} H_i(X;G)$ if $X$ is a compact nonsingular $G$-$\mathcal{AS}$ set,
	\item $\beta_i^G(X) = \beta_i^G(Y) + \beta_i^G(X \setminus Y)$ if $Y \subset X$ is an equivariant closed inclusion,
	\item $\beta_i^G(V) = \beta_i^G( \mathbb{R}^n \times X)$ with $G$ acting diagonally on the right-hand product, $\mathbb{R}^n$ being equipped with the trivial action of $G$, if $V \rightarrow X$ is a $G$-equivariant vector bundle with fiber $\mathbb{R}^n$, i.e. the restriction to $X$ of a vector bundle with fiber $\mathbb{R}^n$ on its Zariski closure $\overline{X}^{Z}$, with a linear $G$-action over the action on $\overline{X}^{Z}$ (this means there exists a finite partition of $\overline{X}^{Z}$ into $G$-globally invariant Zariski constructible sets on which the vector bundle is trivial and the action of $G$ sends linearly a fiber on another).  
\end{enumerate} 
The map $\beta^G_i(\cdot)$ is unique with these properties and is called the $i$-th equivariant virtual Betti number.

For $X$ a $G$-$\mathcal{AS}$ set, we then denote
$$\beta^G(X) := \sum_{i \in \mathbb{Z}} \beta^G_i(X) u^i \in \mathbb{Z}[u][[u^{-1}]]$$
the equivariant virtual Poincar\'e series of $X$.
\end{theo}

\begin{rem} 
\begin{itemize}
	\item For $G = \{e\}$, the equivariant virtual Poincar\'e series is the virtual Poincar\'e polynomial defined in \cite{MCP-VB}.
	\item The assumption of finiteness of the group $G$ is necessary to show the existence of the equivariant virtual Betti numbers. In particular, when $G$ is finite, there always exist an equivariant resolution of singularities (\cite{Villa}, \cite{BM}) and an equivariant compactification (see \cite{DL}).
\end{itemize}
\end{rem}

\begin{ex} \label{exevps} \begin{enumerate}
\item If we consider the sphere $S^2$ equipped with the central symmetry, since $S^2$ is compact nonsingular, we have
$$\beta^G(S^2) = \sum_{i \in \mathbb{Z}} \dim_{\mathbb{Z}_2} H_i(X;G) u^i = u^2 + u + 1$$
(with $G = \mathbb{Z}/2\mathbb{Z}$). If now we consider an action of $G$ on $S^2$ which fixes at least one point, we have
$$\beta^G(S^2) = u^2 + u + \sum_{i \leq 0} 2 u^i = u^2 + u + 2\frac{u}{u-1}$$
(see Example \ref{exsphere}).

\item The equivariant virtual Poincar\'e series of a point is $\frac{u}{u-1}$ and the equivariant virtual Poincar\'e series of two points inverted by an action of $G = \mathbb{Z}/2\mathbb{Z}$ is $1$ : in both cases, the Hochschild-Serre spectral sequence degenerates at $E^2$-term.

\item Let the affine plane $\mathbb{R}^2$ be equipped with an involution $s$. To compute the equivariant virtual Poincar\'e series $\beta^G(\mathbb{R}^2)$ (with $G = \{1,s\}$), consider an equivariant one-point compactification of $\mathbb{R}^2$. It is equivariantly Nash isomorphic to a sphere $S^2$ equipped with an involution fixing at least the point $S^2 \setminus \mathbb{R}^2$. Therefore
$$\beta^G(\mathbb{R}^2) = \beta^G(S^2) - \beta^G(S^2 \setminus \mathbb{R}^2) = \sum_{i \leq 2} u^i = \frac{u^3}{u-1}.$$

\item Let $\mathbb{R}^2$ be equipped with an action of $G := \mathbb{Z}/2\mathbb{Z}$ given by $s : (x,y) \mapsto (\epsilon x , \epsilon' y)$, with $\epsilon, \epsilon' \in \{-1, 1\}$, and $E$ denote the exceptional divisor of the equivariant blowing-up of the plane at $0$. Then, $\beta^G(E) = \beta^G(\mathbb{P}^1) = \beta^G(S^1)$, where the circle $S^1$ is equipped with an involution fixing at least one point. Then $\beta^G(E) = u + 2\frac{u}{u-1}$ (we compute the Hochschild-Serre spectral sequence of $S^1$). 
\end{enumerate}
\end{ex}

Contrary to the virtual Poincar\'e polynomial, we do not know the behaviour of the equivariant virtual Poincar\'e series towards products in general case. Nevertheless, we have the following result regarding the equivariant virtual Poincar\'e series of the product of a $G$-$\mathcal{AS}$ set with an affine space. We will use the following two properties in the proof of Denef-Loeser formula for the equivariant zeta functions (subsection \ref{parequivDLform}).

\begin{prop}[Proposition 3.13 of \cite{GF}] \label{prodaf} Let $X$ be any $G$-$\mathcal{AS}$ set and equip the affine variety $\mathbb{R}^n$ with any algebraic action of $G$. If we equip their product with the diagonal action of $G$, we have
$$\beta^G(\mathbb{R}^n \times X) = u^n \beta^G(X).$$
In particular, $\beta^G(\mathbb{R}^n) = \frac{u^{n+1}}{u-1}$.
\end{prop}

\begin{lem} \label{lemprod} Let $X$ be any $G$-$\mathcal{AS}$ set and equip the real line $\mathbb{R}$ with any algebraic variety action of $G$ stabilizing $0$. Now let $m \in \mathbb{N}^*$ and equip the product $(\mathbb{R}^*)^m \times X$ with the induced diagonal action of $G$. Then we have
$$\beta^G\left((\mathbb{R}^*)^m \times X\right) = (u-1)^m \beta^G(X).$$
\end{lem}

\begin{proof}
We prove this equality by induction on $m$ : we have 
$$\beta^G\left(\mathbb{R}^* \times X\right) = \beta^G\left(\mathbb{R} \times X\right) - \beta^G\left(\{0\} \times X\right) = (u-1)\beta^G(X)$$
by Proposition \ref{prodaf}, and, if we assume the property to be true for a fixed $m \in \mathbb{N}^*$,
$$\beta^G\left((\mathbb{R}^*)^{m+1} \times X\right) = \beta^G\left((\mathbb{R}^*) \times (\mathbb{R}^*)^{m} \times X\right) = (u-1) \beta^G\left((\mathbb{R}^*)^{m} \times X\right) = (u-1)^{m+1} \beta^G(X).$$
\end{proof}

\subsection{Equivariant zeta functions} \label{subsecequivzetafunc}

Consider a linear action of $G$ on $\mathbb{R}^d$, given by linear isomorphisms $\alpha_g$, $g \in G$, and equip $\mathbb{R}$ with the trivial action of $G$. The space $\mathcal{L} = \mathcal{L}(\mathbb{R}^d,0)$ of formal arcs $(\mathbb{R}, 0) \rightarrow (\mathbb{R}^d,0)$ at the origin of $\mathbb{R}^d$ is naturally equipped with the induced action of $G$ given by 
$$g \cdot \gamma := t \mapsto \alpha_g (\gamma(t))$$
for all $g \in G$ and all $\gamma : (\mathbb{R}, 0) \rightarrow (\mathbb{R}^d,0) \in \mathcal{L}$. Notice that, if $\gamma(t) = a_1 t + a_2 t^2 + \ldots$, $g \cdot \gamma(t) = \alpha_g(a_1) t + \alpha_g(a_2) t^2 + \ldots$ by the linearity of the action.

For all $n \geq 1$, thanks to its linearity, the action of $G$ on $\mathcal{L}$ induces an action on the space 
$$\mathcal{L}_n = \mathcal{L}_n(\mathbb{R}^d,0) = \left\{\gamma : (\mathbb{R}, 0) \rightarrow (\mathbb{R}^d,0)~|~\gamma(t) =  a_1 t + a_2 t^2 + \ldots + a_n t^n,~a_i \in \mathbb{R}^d \right\}$$
of arcs truncated at the order $n+1$. Furthermore, the truncation morphism $\pi_n : \mathcal{L} \rightarrow \mathcal{L}_n$ is equivariant with respect to these actions of $G$.
\\

Consider now an equivariant Nash germ $f : (\mathbb{R}^d, 0) \rightarrow (\mathbb{R},0)$, that is $f$ is invariant under right composition with the linear action of $G$. Then, for all $n \geq 1$, the set 
$$A_n(f) := \left\{ \gamma \in \mathcal{L}_n ~|~ f \circ \gamma(t) = c t^n + \cdots,~c \neq 0 \right \}$$ 
of truncated arcs of $\mathcal{L}_n$ becoming series of order $n$ after left composition with $f$ is globally stable under the action of $G$ on $\mathcal{L}_n$. 
 
Consequently, we can apply the equivariant virtual Poincar\'e series to the sets $A_n(f)$, which are Zariski constructible subsets of $\mathbb{R}^{nd}$ equipped with an algebraic action of $G$, and we define the naive equivariant zeta function
$$Z_f^G(u,T) := \sum_{n \geq 1} \beta^G\left(A_n(f)\right) u^{-nd} T^n \in \mathbb{Z}[u][[u^{-1}]][[T]]$$
of $f$
\\

Similarly, the sets
$$A_n^+(f) := \{\gamma \in \mathcal{L}_n~|~f \circ \gamma(t) = + t^n + \cdots \} \mbox{     and     } A_n^-(f) := \{\gamma \in \mathcal{L}_n~|~f \circ \gamma(t) = - t^n + \cdots \}$$
are also stable under the action of $G$ on $\mathcal{L}_n$ and we define the equivariant zeta functions with signs $Z_f^{G,+}$ and $Z_f^{G,-}$ of the invariant Nash germ $f$ :
$$Z_f^{G,\pm}(u,T) := \sum_{n \geq 1} \beta^G\left(A_n^{\pm}(f)\right) u^{-nd} T^n \in \mathbb{Z}[u][[u^{-1}]][[T]].$$

\begin{rem} 
\begin{itemize}
 \item For $G = \{e\}$, the equivariant zeta functions are the zeta functions defined in \cite{GF-MI} and \cite{GF-ZF}.
 \item These equivariant zeta functions are different from the equivariant zeta functions defined in \cite{GF}.
\end{itemize}
\end{rem}

\begin{ex}[see also \cite{KoiPar} and \cite{GF-MI}] Equip the affine line $\mathbb{R}$ with the linear involution $s : x \mapsto -x$. Let $k \in \mathbb{N}^*$ and consider the invariant Nash germ $f : (\mathbb{R},0) \rightarrow (\mathbb{R},0)$ given by $f(x) = x^{2k}$. 

For all $n \geq 1$, if $n$ is not divisible by $2k$, $A_n(f)$ is empty, and if $n = 2k m$,
$$A_n(f) =  \left\{\gamma : (\mathbb{R}, 0) \rightarrow (\mathbb{R},0)~|~\gamma(t) =  a_m t^m + \ldots + a_n t^n,~a_m \neq 0 \right\}$$
is equivariantly Nash isomorphic to $\mathbb{R}^* \times \mathbb{R}^{n-m}$ equipped with the diagonal action of $G := \mathbb{Z}/2\mathbb{Z}$ on each factor induced from the action of $s$ on $\mathbb{R}$. Therefore the equivariant virtual Poincar\'e series of $A_n(f)$ is $(u-1) \frac{u^{n-m+1}}{u-1} = u^{n-m+1}$ if $n = 2k m$ (by Lemma \ref{lemprod} and Proposition \ref{prodaf}), $0$ otherwise and we have
$$Z_f^G(u,T) = \sum_{m \geq 1} u^{2k m-m+1} \left(\frac{T}{u} \right)^{2k m} = \frac{u~T^{2k}}{u - T^{2k}}.$$ 

Now, $f$ is positive so $Z_f^{G,-} = 0$, and for $n = 2k m$,
$$A_n^+(f) =  \left\{\gamma : (\mathbb{R}, 0) \rightarrow (\mathbb{R},0)~|~\gamma(t) =  \pm t^m + \ldots + a_n t^n \right\}$$
is equivariantly Nash isomorphic to $\{\pm 1\} \times \mathbb{R}^{n-m}$, hence $\beta^G(A_n^+(f)) = u^{n-m}$ (the points $-1$ and $+1$ are exchanged by the involution $s$). Thus,
$$Z_f^G(u,T) = \sum_{m \geq 1} u^{2k m-m} \left(\frac{T}{u} \right)^{2k m} = \frac{T^{2k}}{u - T^{2k}}.$$ 
\end{ex}

\subsection{Denef-Loeser formulae for equivariant zeta functions} \label{parequivDLform}

In the following proposition \ref{equivDLform}, we show that, as the non-equivariant one in \cite{GF-MI} and \cite{GF-ZF}, the naive equivariant zeta function is rational. This Denef-Loeser formula for an equivariant modification will allow us to prove that two invariant Nash germs equivariantly blow-Nash equivalent through an equivariant blow-Nash isomorphism have the same naive equivariant zeta function (Theorem \ref{equivzetafuncinv}). 

We keep the notations from previous subsection \ref{subsecequivzetafunc}.

\begin{prop} \label{equivDLform}
Let $\sigma : (M, \sigma^{-1}(0)) \rightarrow (\mathbb{R}^d,0)$ be an equivariant Nash modification of $f$.

At first, we keep notations from the non-equivariant case (\cite{GF-ZF}) :

\begin{itemize}
	\item Let $(f \circ \sigma)^{-1}(0) = \bigcup_{j \in J} E_j$ be the decomposition of $(f \circ \sigma)^{-1}(0)$ into irreducible components. Then there exists $K \subset J$ such that $\sigma^{-1}(0) = \bigcup_{k \in K} E_k$.
	\item Put $N_i := mult_{E_i}~f \circ \sigma$ and $\nu_i := 1 + mult_{E_i}~jac~\sigma$, and, for $I \subset J$, $E_I^0 := \left(\bigcap_{i \in I} E_i\right) \setminus \left(\bigcup_{j \in J \setminus I} E_j \right)$.
\end{itemize}

Now, the action of $G$ on $M$ induces an action of $G$ on the set of irreducible components of $(f \circ \sigma)^{-1}(0)$. For $j_1, j_2 \in J$ and $g \in G$, we write the equality $j_2 = g \cdot j_1$ if $E_{j_2} = g \cdot E_{j_1}$. This induces an action of $G$ on the set $\Lambda$ of non-empty subsets of $J$ and we denote by $\underline{I}$ the orbit of a non-empty subset $I$ of $J$.  

For $\underline{I}$ in $\Lambda/G$, we then denote by $E^0_{\underline{I}}$ the union of the sets $E^0_{g \cdot I} = g \cdot E^0_I = \left(\bigcap_{i \in I} g \cdot E_i \right) \setminus \left(\bigcup_{j \in J \setminus I} g \cdot E_j \right)$, $g \in G$ (it is the orbit of $E^0_I$ in $M$) and we have the equality

$$Z^G_f(u,T) = \sum_{\underline{I} \in \Lambda/G} (u-1)^{|I|} \beta^G\left(E^0_{\underline{I}} \cap \sigma^{-1}(0)\right) \prod_{i \in I} \frac{u^{-\nu_i} T^{N_i}}{1-u^{-\nu_i} T^{N_i}}.$$
\end{prop}

\begin{rem}
For all $i \in I$ and $g \in G$, $mult_{g \cdot E_i} f \circ \sigma = mult_{E_i} f \circ \sigma$ and $mult_{g \cdot E_i} jac~ \sigma = mult_{E_i} jac~ \sigma$, thanks to the equivariance of $f$ and $\sigma$ (see part (iv) of the proof below).
% The first equality comes from the equivariance of $f$, the second one from the definition of an equivariant Nash modification (see section \ref{equivbnequiv}).
\end{rem}

\begin{proof} The proof is a generalization to the equivariant setting of the proof of Denef-Loeser formula in \cite{GF-MI} and \cite{GF-ZF}, which uses the theory of motivic integration on arc spaces for arc-symmetric sets (see also \cite{DL-GA}). The key point is the justification of Kontsevich change of variables formula (\cite{Kon}) in our setting.
\\

The proof runs as follows. We define the notion of $G$-stable subsets of the arc space associated to $(\mathbb{R}^d,0)$ or $(M, \sigma^{-1}(0))$. These sets constitute the measurable sets with respect to a measure defined using the equivariant virtual Poincar\'e series. Here, we use the good behaviour of $\beta^G$ with respect to equivariant vector bundles (Proposition \ref{prodaf}) to justify that this measure is well-defined.

This allows one to define an integration with respect to this equivariant measure. We show the validity of the Kontsevich change of variables in the equivariant setting (Proposition \ref{Konequiv} below) just after the present proof.

This key formula provides us a first intermediate equality for $Z_f^G(u,T)$, bringing out some $\mathcal{AS}$ sets globally invariant under the induced actions of $G$, which involve the equivariant Nash modification $\sigma$ of $f$.

The final step is the computation of the value of the equivariant virtual Poincar\'e series of these $G$-$\mathcal{AS}$ sets in terms of the irreducible components of $(f \circ \sigma)^{-1}(0)$.
\vspace{0.5cm}

\underline{(i) Equivariant measurability and equivariant integration on arc spaces}
\\

We first define a notion of equivariant measurability and equivariant measure in the arc spaces $\mathcal{L}(\mathbb{R}^d, 0)$ and $\mathcal{L}\left(M, \sigma^{-1}(0)\right) = \{\gamma : (\mathbb{R},0) \rightarrow (M, \sigma^{-1}(0)) \mbox{ formal } \}$. The action of $G$ on $M$, given by algebraic isomorphisms $\delta_g$, $g \in G$, induces an action on $\mathcal{L}(M, \sigma^{-1}(0))$ by composition. For all $n \geq 0$, the space $\mathcal{L}_n(M, \sigma^{-1}(0))$ of arcs truncated at order $n+1$ is stable under the action of $G$ on $\mathcal{L}(M, \sigma^{-1}(0))$ and the $n+1$-th order truncating morphisms $\pi_n : \mathcal{L}(M, \sigma^{-1}(0)) \rightarrow \mathcal{L}_n(M, \sigma^{-1}(0))$ is equivariant (see part (iv) of the proof).
%Notice that, if $\gamma \in \mathcal{L}\left(M, \sigma^{-1}(0)\right)$ and $g \in G$, a Nash chart around $\gamma(0)$ gives, by composition with $\delta_g$, a Nash chart around $g \cdot \gamma(0)$ and therefore, if $\gamma(t) = a_1 t + a_2 t^2 \ldots$ (with $a_i \in \mathbb{R}^d$) in the chart around $\gamma(0)$, we have $g \cdot \gamma(t) = a_1 t + a_2 t^2 \ldots$ in the chart around $g \cdot \gamma(0)$ (see part (iv) of the proof for more details).
%In particular, for all $n \geq 0$, the space $\mathcal{L}_n(M, \sigma^{-1}(0))$ of arcs truncated at order $n+1$ is stable under the action of $G$ on $\mathcal{L}(M, \sigma^{-1}(0))$ and the $n+1$-th order truncating morphisms $\pi_n : \mathcal{L}(M, \sigma^{-1}(0)) \rightarrow \mathcal{L}_n(M, \sigma^{-1}(0))$ is equivariant. 

For convenience, in the following definitions, $\mathcal{L}$ will denote either $\mathcal{L}(\mathbb{R}^d, 0)$ or $\mathcal{L}(M, \sigma^{-1}(0))$.

Now we say that a subset $A$ of the arc space $\mathcal{L}$ is $G$-stable if there exists $n \geq 0$ and an $\mathcal{AS}$-subset $C$ of $\mathcal{L}_n$, globally invariant under the algebraic action of $G$ on $\mathcal{L}_n$, such that $A = \pi_n^{-1}(C)$. Notice that a $G$-stable set is globally invariant under the action of $G$ on $\mathcal{L}$. Then we define the measure $\beta^G(A)$ of a $G$-stable set $A$ by setting
$$\beta^G(A) := u^{-(n+1)d} \beta^G(\pi_n(A)) \in \mathbb{Z}[u][[u^{-1}]]$$
for $n$ big enough.
\\

Let us show that this measure is well-defined. This is actually a consequence of the fact that the truncation projections $q_n : \mathcal{L}_{n+1} \rightarrow \mathcal{L}_n$ are vector bundles with fiber $\mathbb{R}^d$, the action of $G$ sending linearly a fiber on another (for $\mathcal{L} = \mathcal{L}(M, \sigma^{-1}(0))$, we can cover the compact set $\sigma^{-1}(0)$ by the orbits of a finite number of open affine subsets $U_{x}$, $x \in \sigma^{-1}(0)$). 

Now, if $A = \pi_n^{-1}(C_n) = \pi_{n+1}^{-1}(C_{n+1})$, since $q_n : C_{n+1} \rightarrow C_n$ is a restriction of the $G$-equivariant vector bundle $q_n : \mathcal{L}_{n+1} \rightarrow \mathcal{L}_n$, we have $\beta^G(C_{n+1}) = u^d \beta^G(C_n)$ by Theorem \ref{eqvirpoiser} and Proposition \ref{prodaf}.
\\

We then define an integral with respect to the measure $\beta^G$ for maps $\theta$ with source a $G$-stable set $A$ and $\mathbb{Z}[u,u^{-1}]$ as target, with finite image and $G$-stable sets as fibers : the integral of $\theta$ over $A$ is
$$\int_A \theta d \beta^G := \sum_{c \in \mathbb{Z}[u,u^{-1}]} c \beta^G\left(\theta^{-1}(c)\right).$$

\vspace{0.5cm}

\underline{(ii) Kontsevich change of variables}
\\

Now we state the equivariant version of the change of variables formula in \cite{Kon} (see also \cite{DL-GA} and \cite{GF-MI}) :

\begin{prop} \label{Konequiv} Let $A$ be $G$-stable set of $\mathcal{L}(\mathbb{R}^d,0)$ and assume that $ord_t~jac~\sigma$ is bounded on $\sigma^{-1}(A)$. Then
$$\beta^G(A) = \int_{\sigma^{-1}(A)} u^{-ord_t~jac~\sigma} d\beta^G.$$
\end{prop}

Here, we denote also by $\sigma$ the equivariant map $\mathcal{L}\left(M, \sigma^{-1}(0)\right) \rightarrow \mathcal{L}(\mathbb{R}^d, 0)~;~\gamma \mapsto \sigma \circ \gamma$. We show Proposition \ref{Konequiv} after the present proof.

\vspace{0.5cm}

\underline{(iii) Applying Kontsevich formula}
\\

We use the equivariant version of Kontsevich formula and the additivity of the equivariant virtual Poincar\'e series to reduce the computation of the naive equivariant zeta function to the computation of the equivariant virtual Poincar\'e series of $G$-$\mathcal{AS}$ sets expressed in terms of the equivariant Nash modification $\sigma$ of $f$.

First, we give notations to the sets that will appear as the proof goes along, similarly to \cite{GF-MI}. For any $n \geq 1$ and $e \geq 1$, we put 
\begin{itemize}
	\item $\mathcal{Z}_n(f) := \pi_n^{-1}(A_n(f))$, 
	\item $\mathcal{Z}_n(f \circ \sigma) := \sigma^{-1}(\mathcal{Z}_n(f))$,
	\item $\Delta_e := \{\gamma \in \mathcal{L}\left(M, \sigma^{-1}(0)\right)~|~ord_t~jac~\sigma(\gamma(t)) = e\}$, 
	\item $\mathcal{Z}_{n,e}(f \circ \sigma) := \mathcal{Z}_n(f \circ \sigma) \cap \Delta_e$.
\end{itemize}
Notice that all the sets $\mathcal{Z}_n(f)$, $\mathcal{Z}_n(f \circ \sigma)$, $\Delta_e$ and $\mathcal{Z}_{n,e}(f \circ \sigma)$ are globally invariant under the actions of $G$ on arc spaces, notably because $\sigma$ is an equivariant Nash modification (see also step (iv) below).
\\

First, since all the sets $\mathcal{Z}_n(f)$ are by definition $G$-stable, we can consider their equivariant measure $\beta^G(\mathcal{Z}_n(f)) = u^{-(n+1)d} \beta^G(A_n(f))$ and write
$$Z^G_f(u,T) = u^d \sum_{n \geq 1} \beta^G(\mathcal{Z}_n(f)) T^n.$$
We then apply the equivariant Kontsevich change of variables formula to compute $\beta^G(\mathcal{Z}_n(f))$ for all $n \geq 1$. Indeed, there exists $c \in \mathbb{N}$ such that for all $n \geq 1$, $\mathcal{Z}_n(f\circ \sigma)$ is the finite disjoint union $\cup_{e \leq cn} \mathcal{Z}_{n,e}(f \circ \sigma)$ (see \cite{GF-MI}) : in particular, for all $n \geq 1$, $ord_t~jac~\sigma$ is bounded on $\mathcal{Z}_n(f\circ \sigma) = \sigma^{-1}(\mathcal{Z}_n(f))$ and we can apply Proposition \ref{Konequiv} to obtain
$$\beta^G(\mathcal{Z}_n(f)) = \int_{\sigma^{-1}(\mathcal{Z}_n(f))} u^{-ord_t~jac~\sigma} d\beta^G = \sum_{e \leq cn} u^{-e} \beta^G\left(\mathcal{Z}_{n,e}(f \circ \sigma)\right).$$

Moreover, if $n \geq 1$ and $e \leq cn$, for any arc $\gamma$ in $\mathcal{Z}_{n,e}(f \circ \sigma)$ (more generally in $\mathcal{L}(M, \sigma^{-1}(0))$), there exists $I \subset J$ such that $\pi_0(\gamma) \in E^0_I \cap \sigma^{-1}(0)$, and more particularly, there exists $\underline{I} \in \Lambda/G$ such that $\pi_0(\gamma) \in E^0_{\underline{I}} \cap \sigma^{-1}(0)$. 

Consequently, we can write the $G$-stable set $\mathcal{Z}_{n,e}(f \circ \sigma)$ as the disjoint union of the sets $\mathcal{Z}_{n,e}^{\underline{I}}(f \circ \sigma) := \mathcal{Z}_{n,e}(f \circ \sigma) \cap \pi_0^{-1}(E^0_{\underline{I}} \cap \sigma^{-1}(0))$, $\underline{I} \in \Lambda/G$, and we have
\begin{eqnarray*}
	\ \beta^G(\mathcal{Z}_{n,e}(f \circ \sigma)) & =  & u^{-(n+1)d} \beta^G\left(\pi_n\left(\sqcup_{\underline{I} \in \Lambda/G}\mathcal{Z}_{n,e}^{\underline{I}}(f \circ \sigma) \right)\right) \\
	& = & u^{-(n+1)d} \sum_{\underline{I} \in \Lambda/G} \beta^G\left(\pi_n\left(\mathcal{Z}_{n,e}^{\underline{I}}(f \circ \sigma) \right)\right)
\end{eqnarray*}
(in the last equality, we used the additivity of the equivariant virtual Poincar\'e series). 
\\

Finally, we have
\begin{equation} \label{DLcalcint}
Z^G_f(u,T) = \sum_{n \geq 1} u^{- nd} T^n \sum_{e \leq c n} u^{-e} \sum_{\underline{I} \in \Lambda/G} \beta^G\left(\pi_n\left(\mathcal{Z}_{n,e}^{\underline{I}}(f \circ \sigma) \right)\right),
\end{equation}
where $\pi_n\left(\mathcal{Z}_{n,e}^{\underline{I}}(f \circ \sigma) \right)$ is the $G$-$\mathcal{AS}$ set
$$\left\{\gamma \in \mathcal{L}_n(M, \sigma^{-1}(0))~|~\gamma(0) \in E^0_{\underline{I}} \cap \sigma^{-1}(0),~ord_t~f \circ \sigma (\gamma(t)) = n,~ord_t~jac~\sigma(\gamma(t)) = e\right\}.$$

\vspace{0.5cm}

\underline{(iv) Computation of $\beta^G\left(\pi_n\left(\mathcal{Z}_{n,e}^{\underline{I}}(f \circ \sigma) \right)\right)$}
\\

Let $n \geq 1$, $e \leq nc$ and $\underline{I} \in \Lambda/G$ (such that $E_I \neq \emptyset$). First, we cover the compact set $E_{I}$ by a finite union of open affine subsets $U_{x_r}$, $x_r \in E_I$, $r = 1, \ldots, m$. We can then write
%extract a finite covering $E_I = \bigcup_{l = 1}^m E_I \cap U_{x_l}$ and this gives us in particular a finite covering of $E^0_{\underline{I}}$ by open $\mathcal{AS}$ sets :
$$E^0_{\underline{I}} = \bigcup_{r = 1}^m \left(  \bigcup_{g \in G} E_{g \cdot I}^0 \cap U_{g \cdot x_r} \right).$$
%= \bigcup_{g \in G} g \cdot \left( \bigcup_{l = 1}^m  \left( E_{I}^0 \cap U_{x_l}\right) \right) = \bigcup_{r = 1}^m \left(  \bigcup_{g \in G} E_{g \cdot I}^0 \cap U_{g \cdot x_r} \right).$$

We are going to compute the equivariant virtual Poincar\'e series of
$$\left\{\gamma \in \mathcal{L}_n(M, \sigma^{-1}(0))~|~\gamma(0) \in E^0_{\underline{I}} \cap \sigma^{-1}(0),~ord_t~f \circ \sigma (\gamma(t)) = n~,~ord_t~jac~\sigma(\gamma(t)) = e\right\}.$$

Consider an arc $\gamma$ in this last set and assume, without any loss of generality, that $\gamma(0) \in E^0_I \cap \sigma^{-1}(0) \cap U_x$ with $x \in \{x_1, \ldots, x_m\}$. For all $i \in I$, we have
\begin{itemize}
	\item $E_i \cap U_x = \varphi_x(\{y_{j_i} = 0\} \cap V_x)$,
	\item $f \circ \sigma(\varphi_x(y_1, \ldots, y_d)) = unit(y_1, \ldots, y_d)~\prod_{i \in I} y_{j_i}^{N_i}$,
	\item $jac~\sigma(\varphi_x(y_1, \ldots, y_d)) = unit(y_1, \ldots, y_d)~\prod_{i \in I} y_{j_i}^{\nu_{i} - 1}$.
\end{itemize}
We can assume $U_x$ and $V_x$ are Nash isomorphic $\mathcal{AS}$-sets such that the above unit Nash functions are nowhere zero on $V_x$ (by intersecting $V_x$ with the orbit under $G$ of the respective complements of their zero sets which are $\mathcal{AS}$-sets).

Let $g \in G$. For $i \in I$, the action of $g$ on $M$ sends an irreducible component $E_i$ locally described in $V_x$ by the equation $y_{j_i} = 0$ on the irreducible component $E_{g \cdot i}$ locally described in $V_{g \cdot x}$ by the same equation $y_{j_{i}} = 0$. Therefore, after same relevant permutations of coordinates in the source and target spaces, the matrix of $\nu_{x,g} : \mathbb{R}^d \supset V_x \rightarrow \mathbb{R}^d \supset V_{g \cdot x}$ becomes of the form
%matrix of $\nu_{x,g} : \mathbb{R}^d \supset V_x \rightarrow \mathbb{R}^d \supset V_{g \cdot x}$ 

$$\left(\vcenter{\xymatrix{
				\lambda_1 \ar@{-}[dddrrr] & 0 \ar@{.}[rr] \ar@{.}[ddrr] &  & 0 \ar@{.}[dd] & 0 \ar@{.}[r] \ar@{.}[ddd] & 0 \ar@{.}[ddd] \\
				0 \ar@{.}[ddrr] \ar@{.}[dd] & & &  & &\\
				&&&0 & &\\
				0 \ar@{.}[rr]& & 0 & \lambda_{|I|} &0 \ar@{.}[r] &0 \\
				0 \ar@{.}[rrr]  \ar@{.}[d] & &   &  0   \ar@{.}[d]            &         \ar@{-}[r] \ar@{-}[d]     & \ar@{-}[d]\\
				0 \ar@{.}[rrr]  &  &   &           0    &     \ar@{-}[r]        & 
				}}
	\right)$$
where all the $\lambda_j$'s are non-zero (recall the definition of an equivariant Nash modification of $f$ in section \ref{equivbnequiv}).

In particular, for all $i \in I$, 
\begin{itemize}
	\item $N_{g \cdot i} = N_{i}$,
	\item $\nu_{g \cdot i} = \nu_{i}$,
	\item if $\gamma(t) = \varphi_x((\gamma_1(t), \cdots, \gamma_d(t)))$ and if we denote $k_{i}(\gamma) := ord_t~\gamma_{j_i}(t)$ for all $i \in I$, $k_{i}(g \cdot \gamma(t)) = k_i(\gamma)$,
\end{itemize}
so that there is an equivariant Nash isomorphism between 
$$\left\{\gamma \in \mathcal{L}_n(M, \sigma^{-1}(0))~|~\gamma(0) \in E^0_{\underline{I}} \cap \sigma^{-1}(0),~ord_t~f \circ \sigma (\gamma(t)) = n,~ord_t~jac~\sigma(\gamma(t)) = e\right\}$$
and
$$\bigsqcup_{k \in A(n,e)} \left\{\gamma \in \mathcal{L}_n(M, \sigma^{-1}(0))~|~\gamma(0) \in E^0_{\underline{I}} \cap \sigma^{-1}(0),~ord_t~\gamma_{j_{i}}(t) = k_i,~i \in I \right\},$$
with $A(n,e) := \{k \in \mathbb{N}^d ~|~\sum_{i \in I} k_i N_{i} = n,~\sum_{i \in I} k_i (\nu_{i} - 1) = e\}$.
Consequently, the equivariant virtual Poincar\'e series of these sets are equal.
\\

%$$\left\{\gamma \in \mathcal{L}_n(M, \sigma^{-1}(0))~|~\gamma(0) \in E^0_{\underline{I}} \cap \sigma^{-1}(0),~\sum_{i \in I} k_i N_{i} = n,~\sum_{i \in I} k_i (\nu_{i} - 1) = e\right\},$$
%so that their equivariant virtual Poincar\'e series are equal.

Let $k \in A(n,e)$. We compute now the equivariant virtual Poincar\'e series of the set
$$W_k := \left\{\gamma \in \mathcal{L}_n(M, \sigma^{-1}(0))~|~\gamma(0) \in E^0_{\underline{I}} \cap \sigma^{-1}(0),~ord_t~\gamma_{j_{i}}(t) = k_i,~i \in I \right\}.$$
We write it as the difference of
$$\left\{\gamma \in \mathcal{L}_n(M, \sigma^{-1}(0))~|~\gamma(0) \in E^0_{\underline{I}} \cap \sigma^{-1}(0),~\gamma_{j_{i}}(t) = c_i t^{k_i} + \cdots, c_i \in \mathbb{R},~ i \in I \right\}$$
and the union over $l \in I$ of the sets
$$\left\{\gamma \in \mathcal{L}_n(M, \sigma^{-1}(0))~|~\gamma(0) \in E^0_{\underline{I}} \cap \sigma^{-1}(0),~\gamma_{j_{l}}(t) = 0 \times t^{k_l} + \cdots \right\}.$$
%$$\left\{\gamma \in \mathcal{L}_n(M, \sigma^{-1}(0))~|~\gamma(0) \in E^0_{\underline{I}} \cap \sigma^{-1}(0),~\gamma_{j_i}(t) = c_i t^{k_i} + \cdots, c_i \in \mathbb{R},~\sum_{i \in I} k_i N_{i} = n,~\sum_{i \in I} k_i (\nu_{i} - 1) = e\right\}$$
%and the union over $l \in I$ of the sets
%$$\left\{\gamma \in \mathcal{L}_n(M, \sigma^{-1}(0))~|~\gamma(0) \in E^0_{\underline{I}} \cap \sigma^{-1}(0),~\gamma_{j_l}(t) = 0 \times t^{k_l} + \cdots,~\sum_{i \in I} k_i N_{i} = n,~\sum_{i \in I} k_i (\nu_{i} - 1) = e\right\}.$$
Thanks to the additivity of the equivariant virtual Poincar\'e series, we are then reduced to compute the equivariant virtual Poincar\'e series of
$$\left\{\gamma \in \mathcal{L}_n(M, \sigma^{-1}(0))~|~\gamma(0) \in E^0_{\underline{I}} \cap \sigma^{-1}(0),~\gamma_{j_{l}}(t) = 0 \times t^{k_l} + \cdots,~l \in \{l_1, \ldots, l_s\} \right\}$$
for any $\{l_1, \ldots, l_s\} \subset I$. Considering the restriction to such a set of the projection $\pi_0$ onto $E^0_{\underline{I}} \cap \sigma^{-1}(0)$, we see that this is a $G$-equivariant vector bundle over $E^0_{\underline{I}} \cap \sigma^{-1}(0)$ with fibers isomorphic to $\mathbb{R}^{|I|-s} \left(\prod_{i \in I} \mathbb{R}^{n-k_i}\right)\left(\mathbb{R}^n\right)^{d-|I|}$.

Therefore, 
\begin{eqnarray*}
\beta^G(W_k) & = & \beta^G(E^0_{\underline{I}} \cap \sigma^{-1}(0)) u^{|I| + nd - \sum_{i \in I} k_i} - \sum_{s \in \{1, \ldots, |I|\}} (-1)^{s+1} \binom{|I|}{s} \beta^G\left(E^0_{\underline{I}} \cap \sigma^{-1}(0) \right) u^{|I|-s + nd - \sum_{i \in I} k_i} \\
& = & (u-1)^{|I|} \beta^G(E^0_{\underline{I}} \cap \sigma^{-1}(0)) u^{nd - \sum_{i \in I} k_i}
\end{eqnarray*}
\\

As a consequence,
% the equivariant virtual Poincar\'e series of
%$$\left\{\gamma \in \mathcal{L}_n(M, \sigma^{-1}(0))~|~\gamma(0) \in E^0_{\underline{I}} \cap \sigma^{-1}(0),~ord_t~f \circ \sigma (\gamma(t)) = n,~ord_t~jac~\sigma(\gamma(t)) = e\right\}$$
%is
%$$\sum_{k \in A(n,e)} (u - 1)^{|I|} \beta^G\left(E^0_{\underline{I}} \cap \sigma^{-1}(0) \right) u^{n d -  \sum_{i \in I} k_i},$$
%that is
\begin{equation} \label{DLcalcpiece}
\beta^G(\pi_n\left(\mathcal{Z}_{n,e}^{\underline{I}}(f \circ \sigma) \right)) = \sum_{k \in A(n,e)} (u - 1)^{|I|} \beta^G\left(E^0_{\underline{I}} \cap \sigma^{-1}(0)\right) u^{nd -  \sum_{i \in I}k_i}.
\end{equation}

\vspace{0.5cm}

\underline{(v) Conclusion of the proof}
\\

Substituting (\ref{DLcalcpiece}) in the equality (\ref{DLcalcint}) of step (iii), we get 
$$Z^G_f(u,T) = \sum_{\underline{I} \in \Lambda/G} (u - 1)^{|I|} \beta^G\left(E^0_{\underline{I}} \cap \sigma^{-1}(0)\right) \sum_{n \geq 1} \sum_{e \leq cn} \sum_{k \in A(n,e)} u^{- e -  \sum_{i \in I} k_i} T^n.$$ 

As in \cite{GF-MI}, we write the sum $\sum_{n \geq 1} \sum_{e \leq cn} \sum_{k \in A(n,e)} u^{- e -  \sum_{i \in  I} k_i} T^n$ as the product $\prod_{i \in I} \frac{u^{-\nu_i} T^{N_i}}{1-u^{-\nu_i} T^{N_i}}$ and we obtain the Denef-Loeser formula

$$Z^G_f(u,T) =  \sum_{\underline{I} \in \Lambda/G} (u - 1)^{|I|} \beta^G\left(E^0_{\underline{I}} \cap \sigma^{-1}(0)\right) \prod_{i \in I} \frac{u^{-\nu_i} T^{N_i}}{1-u^{-\nu_i} T^{N_i}}.$$ 

\end{proof}

%\begin{rem} \label{remcond}The proof still works if we assume that $\sigma$ verifies only conditions 1, 2, 3 and conditions 6 (a), (c) and (g) of the definition of an equivariant Nash modification in section \ref{equivbnequiv}. 
%\end{rem}

We next prove the equivariant version of Kontsevich change of variables (Proposition \ref{Konequiv}), which is a key tool in the demonstration of the above Denef-Loeser formula. In order to achieve this goal, we need the following lemma which is an equivariant Nash analog to Lemma 3.4 of \cite{DL-GA} (see also Lemma 2.11 of \cite{GF-ZF}) :

\begin{lem} \label{lemequivKon}
Let $h : (M, h^{-1}(0)) \rightarrow (\mathbb{R}^d,0)$ be an equivariant proper surjective Nash map such that the action of $G$ on $M$ is locally linear around $h^{-1}(0)$, i.e corresponds locally to a linear $G$-action, in $G$-globally invariant affine open charts of $M$.

%	\item the action of $G$ on $M$ preserves globally each irreducible component of $h^{-1}(0)$,

%and assume that the action of $G$ on $M$ preserves the multiplicities of the jacobian determinant of $h$ along its exceptional divisors.

%such that each irreducible component of $h^{-1}(0)$ is globally preserved by the action of $G$ on $M$ that we assume to be locally linear around $h^{-1}(0)$.
% assume that the action of $G$ on $M$ preserves the multiplicities of the jacobian determinant of $h$ along its exceptional divisors.

For all $e \geq 1$, set
$$\Delta_e := \left\{\gamma \in \mathcal{L}(M, h^{-1}(0))~|~ord_t~jac~h(\gamma(t)) = e \right\}$$
and for all $n \geq 1$,
$$\Delta_{e,n} := \pi_n(\Delta_e),$$
and denote by $h_n$ the equivariant map $\pi_n \circ h :  \mathcal{L}_n(M,h^{-1}(0)) \rightarrow \mathcal{L}_n(\mathbb{R}^d,0)$.

If $n \geq 2e$, then $h_n(\Delta_{e,n})$ is an $\mathcal{AS}$ set, globally invariant under the action of $G$ on $\mathcal{L}_n(\mathbb{R}^d,0)$, and $h_n$ is an equivariantly piecewise trivial fibration over $h_n(\Delta_{e,n})$, with $G$-globally invariant $\mathcal{AS}$ sets as pieces, with fiber $\mathbb{R}^e$ (more precisely $h_n^{-1}(h_n(\Delta_{e,n})) \rightarrow h_n(\Delta_{e,n})$ is a $G$-equivariant vector bundle with fiber $\mathbb{R}^e$).
\end{lem}

\begin{rem} The details of this fibration are given at the end of the proof below.
\end{rem}

\begin{proof}[Proof of Lemma \ref{lemequivKon}]
Fix $e \geq 1$ and $n \geq 2e$.

The fact that $h_n(\Delta_{e,n})$ is an $\mathcal{AS}$ set is given by the non-equivariant result of G. Fichou in \cite{GF-ZF}, Lemma 2.11. Since $\Delta_{e,n}$ is globally invariant under the action of $G$ on $\mathcal{L}_n(M, h^{-1}(0))$, so is the set $h_n(\Delta_{e,n})$ under the action of $G$ on $\mathcal{L}_n(\mathbb{R}^d,0)$ because $h_n$ is equivariant.
\\

In the proof of the second assertion of the lemma, we refer to the proof of Lemma 3.4 in \cite{DL-GA}, replacing the terms ``regular maps'' by ``Nash maps'' and ``constructible sets'' by ``$\mathcal{AS}$ sets'' (see also \cite{GF-ZF}). 

Consider an equivariant section $s : \mathcal{L}_n(\mathbb{R}^d,0) \rightarrow \mathcal{L}(\mathbb{R}^d,0)$ of $\pi_n$. We have $s(h_n(\Delta_{e,n})) \subset h(\Delta_e)$ and $h^{-1}$ is well-defined on $h(\Delta_e)$ (because $h(\gamma_1) \neq h(\gamma_2)$ if $\gamma_1 \neq \gamma_2$ with $\gamma_1 \in \Delta_e$), Nash and equivariant, so one can construct the equivariant mapping 
$$\theta : h_n(\Delta_{e,n}) \rightarrow \Delta_e~;~ \gamma \mapsto h^{-1}(s(\gamma)).$$ 
It is an equivariantly piecewise morphism : this means there exists a finite partition of the domain of $\theta$ into $\mathcal{AS}$ sets globally invariant under the action of $G$ on $\mathcal{L}_n(\mathbb{R}^d,0)$, such that the restriction of $\theta$ to each piece is an equivariant Nash map, that is induced by an equivariant semialgebraic and analytic map. 
%This can be shown by considerations similar to the ones in the proof of Proposition \ref{equivDLform}. 

%It is a piecewise morphism : this means there exists a finite partition of the domain of $\theta$ into $\mathcal{AS}$ sets, such that the restriction of $\theta$ to each $\mathcal{AS}$ piece is a Nash map. 

%Now, let $R_1,\ldots, R_m$ be these $\mathcal{AS}$ pieces and, for each $i \in \{1, \ldots, l\}$, consider the orbit of $R_i$ under the action of $G$ on $\mathcal{L}_n(\mathbb{R}^d,0)$ from which we retrieve the orbit of $R_1 \cup \cdots \cup R_{i-1}$ : we denote the obtained $G$-$\mathcal{AS}$ set by $\widetilde{R_i}$. 

%Let $i \in \{1, \ldots, l\}$ and denote by $\psi_i$ the restriction of $\theta$ to the piece $R_i$ : it is a Nash map by assumption. We extend $\psi_i$ to an equivariant Nash map $\widetilde{\psi_i}$ on the $G$-orbit of $R_i$ by setting $\widetilde{\psi_i}(g \cdot \gamma) := \delta_g \circ \psi_i(\gamma)$ if $\gamma \in R_i$ and $g \in G$ : if $g \cdot \gamma \in R_i$, we have $\delta_g \circ \psi_i(\gamma) = \psi_i(g \cdot \gamma)$ because $\theta$ is equivariant. We then denote again by $\widetilde{\psi_i}$ the restriction of $\widetilde{\psi_i}$ to $\widetilde{R_i}$.

%This shows that we can assume $\theta$ to be an equivariantly piecewise morphism, that is there exists a finite partition of the domain of $\theta$ into $\mathcal{AS}$ sets globally invariant under the action of $G$ on $\mathcal{L}_n(\mathbb{R}^d,0)$, such that the restriction of $\theta$ to each piece is an equivariant Nash map.
 
One can then use the map $\theta$ to express the fiber of an arc $\gamma$ of $h_n(\Delta_{e,n})$ under $h_n$ : 
$$h_n^{-1}(\gamma) = \{\theta(\gamma) + t^{n+1-e} \gamma' \mbox{ mod } t^{n+1}~|~ \gamma' \mbox{ formal and } (Jac~h(\theta(\gamma)))\gamma' \equiv 0 \mbox{ mod } t^e \},$$
which can be identified to a linear subspace of $\mathbb{R}^{de}$ of dimension $e$. Furthermore, the action of $g \in G$ sending the fiber $h_n^{-1}(\gamma)$ on the fiber $h_n^{-1}(g \cdot \gamma)$ is given by the matrix $A_{x,g}$ for some $x \in h^{-1}(0)$. %Finally, these identifications are given by a equivariantly piecewise morphism.
  
%each fiber can be identified with $\mathbb{R}^e$, this identification being a piecewise morphism on a refinement of the partition of the domain of the equivariantly piecewise morphism $\theta$. 

%The identification is furthermore an equivariantly piecewise morphism because, if the fiber of an arc $\gamma$ is expressed in local coordinates around $\theta(\gamma)(0)$, for any $g \in G$, these local coordinates are also local coordinates around $\theta(g \cdot \gamma)(0) = g \cdot \theta(\gamma)(0)$ and the fiber of $g \cdot \gamma$ is obtained from the fiber of $\gamma$ by applying the identity matrix (see proof of Proposition \ref{equivDLform}).

Therefore, there exists a finite partition of $h_n(\Delta_{e,n})$ into globally $G$-invariant $\mathcal{AS}$ subsets $(S_i)_{i = 1, \ldots, m}$ of $\mathcal{L}_n(\mathbb{R}^d,0)$, such that for any $i \in \{1, \ldots, m\}$, $h_n^{-1}(S_i)$ is a $G$-$\mathcal{AS}$ subset of $\mathcal{L}_n(M, h^{-1}(0))$, Nash isomorphic to $S_i \times \mathbb{R}^e$, the action of $G$ sending linearly a fiber on another. 
%$G$ acting diagonally on this product, the action of $G$ on $S_i$ being induced from the one on $\mathcal{L}_n(\mathbb{R}^d,0))$, and the one on the right-hand affine space being linear, with ${h_n}_{|h_n^{-1}(S_i)}$ corresponding under the equivariant Nash isomorphism to the equivariant projection $S_i \times \mathbb{R}^e \rightarrow \mathbb{R}^e$.

\end{proof}

\begin{proof}[Proof of Proposition \ref{Konequiv}]

We are now ready to prove the equivariant Kontsevich change of variables formula. We are going to compute the integral against the measure $\beta^G$ of the map 
$$\zeta : \sigma^{-1}(A) \rightarrow \mathbb{Z}[u,u^{-1}]~;~ \gamma \mapsto u^{-ord_t~jac~\sigma(\gamma)}$$
over $\sigma^{-1}(A)$, and show that it equals $\beta^G(A)$.
\\

We have
\begin{eqnarray*}
\int_{\sigma^{-1}(A)} \zeta d \beta^G & = & \sum_{c \in \mathbb{Z}[u,u^{-1}]} c \beta^G( \zeta^{-1}(c))  \mbox{~~~~(by definition of the integral)} \\
                                      & = & \sum_{1 \leq e \leq \rho} u^{-e} \beta^G( \sigma^{-1}(A) \cap \Delta_e)  \mbox{~~~~($ord_t~jac~\sigma(\gamma)$ is bounded on $\sigma^{-1}(A)$)}\\
                                      & = & \sum_{1 \leq e \leq \rho} u^{-e} u^{-(n+1)d} \beta^G(\pi_n(\sigma^{-1}(A) \cap \Delta_e))  \mbox{~~~~(for $n$ big enough and bigger than $2 \rho$)} \\
                                      & = & \sum_{1 \leq e \leq \rho} u^{-e} u^{-(n+1)d} \beta^G(\pi_n(\sigma^{-1}(A)) \cap \Delta_{e,n})  \mbox{~~~~($\sigma^{-1}(A)$ is stable)}
\end{eqnarray*}

Now fix $1 \leq e \leq \rho$. We have the equality of sets 
$$\pi_n(\sigma^{-1}(A)) \cap \Delta_{e,n} = \sigma_n^{-1}(\pi_n(A)) \cap \sigma_n^{-1}(\sigma_n(\Delta_{e,n})) = \sigma^{-1}\left(\pi_n(A) \cap \sigma_n(\Delta_{e,n})\right),$$ 
where $\sigma_n = \pi_n \circ \sigma$. The equality $\pi_n(\sigma^{-1}(A)) = \sigma_n^{-1}(\pi_n(A))$ comes from the stability of $A$ and the fact that $\pi_n \circ \sigma \circ \pi_n = \pi_n \circ \sigma$ on $\mathcal{L}(M, h^{-1}(0))$, and we use Lemma 2.12 of \cite{GF-ZF} to show that $\Delta_{e,n} = \sigma_n^{-1}(\sigma_n(\Delta_{e,n}))$ (recall that $n \geq 2e$).         

We then compute $\beta^G\left(\sigma_n^{-1}\left(\pi_n(A) \cap \sigma_n\left(\Delta_{e,n}\right)\right)\right)$ using the fact that, by previous Lemma \ref{lemequivKon}, $\sigma_n$ is an equivariantly piecewise trivial fibration over $\sigma_n(\Delta_{e,n})$ and more precisely that $\sigma_n^{-1}\left(\pi_n(A) \cap \sigma_n\left(\Delta_{e,n}\right)\right) \rightarrow \pi_n(A) \cap \sigma_n\left(\Delta_{e,n}\right)$ is a restriction of the $G$-equivariant vector bundle $\sigma_n^{-1}\left(\sigma_n\left(\Delta_{e,n}\right)\right) \rightarrow \sigma_n\left(\Delta_{e,n}\right)$ with fiber $\mathbb{R}^e$, so that 
%Keeping the same notations as in the end of the proof of Lemma \ref{lemequivKon}, we have 

$$\beta^G\left(\sigma_n^{-1}\left(\pi_n(A) \cap \sigma_n\left(\Delta_{e,n}\right)\right)\right) = u^e \beta^G\left(\pi_n(A) \cap \sigma_n(\Delta_{e,n})\right).$$

%\begin{eqnarray*}
%\beta^G\left(\sigma_n^{-1}\left(\pi_n(A) \cap \sigma_n\left(\Delta_{e,n}\right)\right)\right) & = & \sum_{i = 1}^m \beta^G\left(\sigma_n^{-1}(\pi_n(A) \cap S_i)\right) \mbox{ ~~~~(additivity of $\beta^G$)} \\
																			%											 & = & \sum_{i = 1}^m \beta^G\left((\pi_n(A) \cap S_i) \times \mathbb{R}^e\right) \\
																			%											 & = & \sum_{i = 1}^m u^e \beta^G(\pi_n(A) \cap S_i) \mbox{~~~~(Proposition \ref{prodaf}) } \\
																			%											 & = & u^e \beta^G\left(\pi_n(A) \cap \sigma_n(\Delta_{e,n})\right) \mbox{~~~~ (additivity of $\beta^G$)}
%\end{eqnarray*}

Consequently,
\begin{eqnarray*}
\int_{\sigma^{-1}(A)} \zeta d \beta^G & = & \sum_{1 \leq e \leq \rho} u^{-e} u^{-(n+1)d} \beta^G\left(\sigma_n^{-1}(\pi_n(A) \cap \sigma_n(\Delta_{e,n}))\right) \\
																			& = & \sum_{1 \leq e \leq \rho} u^{-(n+1)d} \beta^G\left(\pi_n(A) \cap \sigma_n(\Delta_{e,n})\right) \\
																			& = & u^{-(n+1)d} \beta^G\left( \pi_n(A) \cap (\sqcup_{1 \leq e \leq \rho} \sigma_n(\Delta_{e,n})) \right) \mbox{~~(the sets $\pi_n(\sigma(\Delta_{e}))$ are disjoint since $n > e$)} \\
																			&   & \\
																			& = & u^{-(n+1)d} \beta^G\left(\pi_n(A)\right) = \beta^G(A)
\end{eqnarray*}
Notice that we used the surjectivity of the map $\sigma_n : \mathcal{L}_n(M, \sigma^{-1}(0)) \rightarrow \mathcal{L}_n(\mathbb{R}^d,0)$, which comes from the arc lifting property of a real modification : see for instance \cite{FukuiPau}.
\end{proof}

Next, we state the Denef-Loeser formula for the equivariant zeta functions with signs. As in the non-equivariant case (see \cite{GF-MI} and \cite{GF-ZF}), we have to consider coverings of the spaces $E^0_{I}$. However, in our equivariant setting, it is necessary to consider coverings of the orbits of these spaces under the induced action of $G$.

\begin{prop} \label{equivDLformsigns} Keep the notations and assumptions of Proposition \ref{equivDLform}. We can write the equivariant zeta functions with signs of $f$ as a rational fraction in terms of its equivariant Nash modification $\sigma$. Precisely, we have the formula
$$Z^{G,\pm}_f(u,T) = \sum_{\underline{I} \in \Lambda/G} (u-1)^{|I|-1} \beta^G\left(\widetilde{E^{0,\pm}_{\underline{I}}} \cap \sigma^{-1}(0)\right) \prod_{i \in I} \frac{u^{-\nu_i} T^{N_i}}{1-u^{-\nu_i} T^{N_i}},$$
where, for $\underline{I} \in \Lambda/G$, $\widetilde{E^{0,+}_{\underline{I}}}$ and $\widetilde{E^{0,-}_{\underline{I}}}$ are $G$-coverings of $E^{0}_{\underline{I}}$. 
%, is the orbit of the covering $\widetilde{E^{0,+}_{I}}$, respectively $\widetilde{E^{0,-}_{I}}$, of $E^{0}_{I}$ under the induced action of $G$. 
\end{prop}   

\begin{rem}
We define the spaces $\widetilde{E^{0,\pm}_{\underline{I}}}$ in the proof below, making precise how the action of $G$ on $M$ induces an action on them.
\end{rem}

\begin{proof}
To prove the Denef-Loeser formula for equivariant zeta functions with signs, we follow the same first steps as in the proof of Proposition \ref{equivDLform}, and we are led to write
$$Z^{G,\pm}_f(u,T) = \sum_{n \geq 1} u^{- nd} T^n \sum_{e \leq c n} u^{-e} \sum_{\underline{I} \in \Lambda/G} \beta^G\left(\pi_n\left(\mathcal{Z}_{n,e}^{\pm,\underline{I}}(f \circ \sigma) \right)\right),$$
where each $\pi_n\left(\mathcal{Z}_{n,e}^{\pm,\underline{I}}(f \circ \sigma) \right)$ is the globally invariant set
$$\left\{\gamma \in \mathcal{L}_n(M, \sigma^{-1}(0))~|~\gamma(0) \in E^0_{\underline{I}} \cap \sigma^{-1}(0),~f \circ \sigma (\gamma(t)) = \pm t^n + \cdots ,~ord_t~jac~\sigma(\gamma(t)) = e\right\}.$$

%that is the orbit of the set
%$$\left\{\gamma \in \mathcal{L}_n(M, \sigma^{-1}(0))~|~\gamma(0) \in E^0_{I} \cap \sigma^{-1}(0),~f \circ \sigma (\gamma(t)) = \pm t^n + \cdots ,~ord_t~jac~\sigma(\gamma(t)) = e\right\},$$
%that we denote by $\pi_n\left(\mathcal{Z}_{n,e}^{\pm,I}(f \circ \sigma) \right)$, under the action of $G$ on $\mathcal{L}_n(M, \sigma^{-1}(0))$.
%\\

%In order to compute this quantity, we look for some equivariant Nash isomorphism given by the above conditions of value, orders and coefficients, the last two being expressed in local coordinates. Since we are dealing with relations involving arcs' coefficients in local coordinates, we will need to consider the gluing of the open sets on which are defined the associated local charts, taking into account the way $G$ is acting on $M$.
%\\

With the same notations as in the proof of Proposition \ref{equivDLform}, let $k \in A(n,e)$ and let $\gamma \in \pi_n\left(\mathcal{Z}_{n,e}^{\pm,\underline{I}}(f \circ \sigma) \right)$ such that $\gamma(0) \in E_I \cap U_x$ and for $i \in I$, $ord_t~\gamma_{j_{i}}(t) = k_i$. The condition $f \circ \sigma (\gamma(t)) = \pm t^n + \cdots$ can be expressed as 
$$u_x\left(\varphi_x^{-1}\left(\gamma(0)\right)\right) \prod_{i \in I} \rho_{j_i}^{N_i} = \pm 1,$$
where $\rho_{j_i}$ is the term of order $k_{i}$ in $\gamma_{j_i}(t)$. Denote
$$W^{\pm}_{I, U_x, \varphi_x} = \left\{(z, \rho) \in (E_I^0 \cap U_x) \times (\mathbb{R}^*)^{|I|}~|~u_x\left(\varphi_x^{-1}(z)\right) \prod_{i \in I} \rho_{j_i}^{N_i} = \pm 1\right\}$$
and notice that, since for $g \in G$, $f \circ \sigma \circ \delta_g = f \circ \sigma$, we have
\begin{eqnarray*}
u_x(y_1, \ldots, y_d) \prod_{i \in I} y_{j_i}^{N_i} 
%& = & f \circ \sigma ( \varphi_x(y_1, \ldots, y_d)) \\
%& = & f \circ \sigma \circ \delta_g ( \varphi_x(y_1, \ldots, y_d)) \\
& = & f \circ \sigma( \varphi_{g \cdot x}( \nu_{x,g} (y_1, \ldots, y_d))) \\
%& = & u_{g \cdot x} ( \nu_{x,g}(y_1, \ldots, y_d)) \prod_{i \in I} y_{j_{g \cdot i}}^{N_i} ( \nu_{x,g} (y_1, \ldots, y_d)) \\
& = & u_{g \cdot x} ( \nu_{x,g}(y_1, \ldots, y_d)) \prod_{i \in I} (\lambda_i y_{j_{i}})^{N_i} \\
%& = & \left(\prod_{i \in I} \lambda_i ^{N_i} \right) u_{g \cdot x} ( \nu_{x,g}(y_1, \ldots, y_d)) \prod_{i \in I} y_{j_{i}}^{N_i}
\end{eqnarray*}
where the constants $\lambda_i$, $i \in I$, are given by the matrix $A_{x,g}$. Therefore,
$$u_{g \cdot x} ( \nu_{x,g}(y_1, \ldots, y_d)) = \frac{1}{\left(\prod_{i \in I} \lambda_i ^{N_i} \right)} u_x(y_1, \ldots, y_d)$$
and in particular, the action of $g \in G$ on $M$ sends $W^{\pm}_{I, U_x, \varphi_x}$ on $W^{\pm}_{g \cdot I, U_{g \cdot x}, \varphi_{g \cdot x}}$.
\\

As a consequence, there is an equivariant Nash isomorphism between $\pi_n\left(\mathcal{Z}_{n,e}^{\pm,\underline{I}}(f \circ \sigma) \right)$ and the gluing of the sets
%, along the spaces $\bigcup_{g \in G} E^0_{g \cdot I} \cap U_{g \cdot x_r}$, $r = 1, \ldots, m$, of the gluings of the sets
$$\bigsqcup_{k \in A(n,e)} W^{\pm}_{g \cdot I, U_{g \cdot x_r}, \varphi_{g \cdot x}} \times \left(\prod_{i \in I} \mathbb{R}^{n-k_i}\right)\left(\mathbb{R}^n\right)^{d-|I|},$$
along the spaces $E^0_{g \cdot I} \cap U_{g \cdot x_r}$, $g \in G$, $r = 1, \ldots, m$.
\\

Thanks to the additivity of the equivariant virtual Poincar\'e series and Lemma \ref{lemprod}, we are then reduced to compute the equivariant virtual Poincar\'e series of the orbit of the set $W^{\pm}_{I, U_{x}, \varphi_{x}}$ (notice that, for $g_1,g_2 \in G$, the sets $\left(E^0_{g_1 \cdot I} \cap U_{g_1 \cdot x}\right)$ and $\left(E^0_{g_2 \cdot I} \cap U_{g_2 \cdot x}\right)$ are equal or do not intersect).
\\

Now, consider the isomorphism from $W^{\pm}_{I, U_x, \varphi_x}$ to $R^{\pm}_{I, U_x, \varphi_x} \times (\mathbb{R}^*)^{|I|-1}$ given in the proof of Proposition 3.5 of \cite{GF-MI}, with
$$R^{\pm}_{I, U_x, \varphi_x} := \left\{(z,t) \in E_I^0 \cap U_x \times \mathbb{R}~|~t^m = \pm \frac{1}{u_x\left(\varphi_x^{-1}(z)\right)} \right\},$$
where $m$ is the greatest common divisor of the $N_i$'s, $i \in I$. It is defined using integers $n_i$, $i \in I$ such that $\sum_{i \in I} n_i N_i = m$.
% (there is at least one of the integers $\frac{N_i}{m}$ which is odd). 

These isomorphisms are compatible with the action of $g \in G$ sending the element $(z,t,\kappa)$ of $R^{\pm}_{I, U_x, \varphi_x} \times (\mathbb{R}^*)^{|I|-1}$ to $(\delta_g(z), \frac{t}{\prod_{i \in I} \lambda_i^{N_i/m}}, g \cdot \kappa) \in R^{\pm}_{g \cdot I, U_{g \cdot x}, \varphi_{g \cdot x}} \times (\mathbb{R}^*)^{|I|-1}$, the $j$-th coordinate of $\kappa$ being sent to itself times $\left(\prod_{i \in I} \lambda_i^{N_i/m}\right)^{-n_j} \lambda_j$.
 
%Finally, the orbit of the set $W^{\pm}_{I, U_x, \varphi_x}$ can be identified with the orbit of $R^{\pm}_{I, U_x, \varphi_x}$ times $(\mathbb{R}^*)^{|I|-1}$.

Consequently, the equivariant virtual Poincar\'e series of the orbit of the set $W^{\pm}_{I, U_{x}, \varphi_{x}}$ is, using Lemma \ref{lemprod}, $(u-1)^{|I|-1}$ times the equivariant virtual Poincar\'e series of the orbit of the set $R^{\pm}_{I, U_x, \varphi_x}$.
\\

Thus, 
$$\beta^G\left(\pi_n\left(\mathcal{Z}_{n,e}^{\pm,\underline{I}}(f \circ \sigma) \right)\right) =  \sum_{k \in A(n,e)} (u - 1)^{|I|-1} \beta^G\left( \widetilde{E^{0,\pm}_{\underline{I}}}\right) u^{nd -  \sum_{i \in I}k_i}$$
where $\widetilde{E^{0,\pm}_{\underline{I}}}$ is the gluing of the sets $R^{\pm}_{g \cdot I, U_{g \cdot x_r}, \varphi_{g \cdot x_r}}$ along the spaces $E^0_{g \cdot I} \cap U_{g \cdot x_r}$, $g \in G$, $r = 1, \ldots, m$.
\\
%the gluing of the sets $W^{\pm}_{I, U, \varphi}$ along the sets $E_I^0 \cap U$ is $(u-1)^{|I|-1} \beta^G\left(\widetilde{E^{0,\pm}_{\underline{I}}}\right)$ (we use Lemma \ref{lemprod}), where $\widetilde{E^{0,\pm}_{\underline{I}}}$ is by definition the orbit under $G$ of the gluing $\widetilde{E^{0,\pm}_{I}}$ of the sets $R^{\pm}_{I, U, \varphi}$ along the sets $E_I^0 \cap U$.
%\\

Now, the end of the computation is the same as in step (v) of the proof of Proposition \ref{equivDLform} (notice that in the above arguments, we dropped the intersection with $\sigma^{-1}(0)$ for the sake of readability).

\end{proof}

\begin{rem} 
%\item Here, it is necessary to assume that $\sigma$ verifies all the conditions of the definition of an equivariant Nash modification.

We can also define the naive equivariant zeta function of the germ of an equivariant Nash function $f : (\mathbb{R}^d,0) \rightarrow (\mathbb{R},0)$ where the affine spaces $\mathbb{R}^d$ and $\mathbb{R}$ are both equipped with a linear action of $G$. Indeed, if $g \mapsto \kappa_g$ denotes the linear action of $G$ on $\mathbb{R}$, then, since~$G$ is finite, for all $g \in G$, $\kappa_g = \pm {\rm Id}_{\mathbb{R}}$. Therefore, the spaces of arcs $A_n(f)$ of $f$ are globally stable under the action of $G$ on $\mathcal{L}$. Furthermore, Denef-Loeser formula of Proposition \ref{equivDLform} is also valid for the naive equivariant zeta function of $f$.

In order to define equivariant zeta functions with signs for $f$, we have to consider the kernel~$H$ of the group morphism $g \mapsto \kappa_g$. Then the arc spaces $A_n^+(f)$ and $A_n^-(f)$ are globally stable under the restricted action of $H$ on $\mathcal{L}$ and we can define the equivariant zeta functions with signs of $f$ with respect to $H$, for which we have Denef-Loeser formula (Proposition \ref{equivDLformsigns}). 

We can also consider, for all $n \geq 1$, the equivariant virtual Poincar\'e series of the reunion of $A_n^+(f)$ and $A_n^-(f)$, which is a $G$-$\mathcal{AS}$-set, and gather this data into a new zeta function. In a subsequent work, we will study this zeta function, as well as its relation to the other equivariant zeta functions of such an equivariant Nash germ $f$.  

\end{rem}

\section{Equivariant zeta functions and equivariant blow-Nash equivalence} \label{secezfebne}

Denef-Loeser formulae for the equivariant zeta functions allow us to show that these latter are invariants for equivariant blow-Nash equivalence via equivariant blow-Nash isomorphisms :

\begin{theo} \label{equivzetafuncinv}
Let $\mathbb{R}^d$ be equipped with a linear action of $G$ and let $f$, $h : (\mathbb{R}^d,0) \rightarrow (\mathbb{R},0)$ be two invariant Nash germs. If $f$ and $h$ are $G$-blow-Nash equivalent via an equivariant blow-Nash isomorphism, then
$$Z^G_f(u,T) = Z^G_h(u,T) \mbox{~~~~and~~~~} Z^{G,\pm}_f(u,T) = Z^{G,\pm}_h(u,T)$$
\end{theo}

\begin{proof}
Let us keep the notations of the definition \ref{defequivbln} of $G$-blow-Nash equivalence. 

%We can compose the equivariant Nash modification $\sigma_f$, respectively $\sigma_h$, with a sequence of equivariant blowings-up with $G$-globally preserved smooth Nash centers and exceptional divisors having only normal crossings with the ones of $\sigma_f$, resp. $\sigma_h$, to obtain a new equivariant Nash modification $\widetilde{\sigma}_f : (\widetilde{M}_f , \widetilde{\sigma}_f^{-1}(0)) \rightarrow (\mathbb{R}^d,0)$ of $f$, resp. $\widetilde{\sigma}_h : (\widetilde{M}_h , \widetilde{\sigma}_h^{-1}(0)) \rightarrow (\mathbb{R}^d,0)$ of $h$, so that the jacobian determinant $jac~\widetilde{\sigma}_f$, resp. $jac~\widetilde{\sigma}_h$, has only normal crossings. This procedure provides an equivariant Nash isomorphism $\widetilde{\Phi} : (\widetilde{M}_f , \widetilde{\sigma}_f^{-1}(0)) \rightarrow (\widetilde{M}_h , \widetilde{\sigma}_h^{-1}(0))$, inducing the equivariant blow-Nash isomorphism $\Phi$ and the equivariant blow-Nash homeomorphism $\phi$ and preserving the multiplicities of the jacobian determinants of $\widetilde{\sigma}_f$ and $\widetilde{\sigma}_h$ along their exceptional divisors.
%\\

We then apply Proposition \ref{equivDLform}, respectively Proposition \ref{equivDLformsigns}, to both $f$ and $h$ and the expressions of the naive equivariant zeta functions, resp. equivariant zeta functions with signs, of $f$ and $h$ given by the Denef-Loeser formula are equal because
\begin{itemize}
	\item $\Phi$ sends the irreducible components of $(f \circ \sigma_f)^{-1}(0)$ onto the irreducible components of $(h \circ \sigma_h)^{-1}(0)$,
	\item the equivariant virtual Poincar\'e series is invariant under equivariant Nash isomorphisms,
	\item the multiplicities $N$ are preserved by $\Phi$ thanks to the commutativity of the diagram defining $G$-blow-Nash equivalence (see Definition \ref{defequivbln}), and the multiplicities $\nu$ are preserved by $\Phi$ because it is an equivariant blow-Nash isomorphism.
\end{itemize}

\end{proof}

%\begin{rem} If, in the definition of an equivariant Nash modification in section \ref{equivbnequiv}, we retrieve conditions 4, 5 and conditions 6 (b), (d), (e) and (f), the naive equivariant zeta function is still an invariant of the equivariant blow-Nash equivalence via an equivariant blow-Nash isomorphism (see remark \ref{remcond}).
%\end{rem}

\begin{ex} \label{exgermnonequivbln} Consider the affine plane $\mathbb{R}^2$ with coordinates $(x,y)$, equipped with the $\mathbb{Z}/2\mathbb{Z}$-action $(x,y) \mapsto (-x,y)$. Let $f$ and $h$ be the Nash germs at $(0,0)$ defined by 
$$f(x,y) = y^4 - x^2~~~~;~~~~ h(x,y) = x^4 - y^2.$$
They are Nash equivalent via the Nash isomorphism $\Phi : \mathbb{R}^2 \rightarrow \mathbb{R}^2~;~(x,y) \mapsto (y,x)$. In particular, they are blow-Nash equivalent via a blow-Nash isomorphism and consequently 
$$Z_f(u,T) = Z_h(u,T)$$
(the zeta functions are invariants for blow-Nash equivalence via blow-Nash isomorphisms : see \cite{GF-ZF}).

However, notice that $\Phi$ is not equivariant with respect to the considered action of $G := \mathbb{Z}/2\mathbb{Z}$ on $\mathbb{R}^2$. We compute the naive equivariant zeta functions of the invariant Nash germs $f$ and $h$ and show that they are not $G$-blow-Nash equivalent via an equivariant blow-Nash isomorphism, using Theorem \ref{equivzetafuncinv}. 
\\

Let $\sigma_1$ be the equivariant blowing-up of $\mathbb{R}^2$ at the origin and let $(\mathbb{R}^2,(X,Y))$ be the chart of the blowing-up in which $\sigma_1$ is given by $\sigma_1(X,Y) = (X Y,Y)$, the action of $G$ on this chart being given by $(X,Y) \mapsto (-X,Y)$. In this chart, we have $f \circ \sigma_1 (X,Y) = Y^2 (Y^2 - X^2)$.

By a second equivariant blowing-up $\sigma_2$ given by $\sigma_2(W,Z) = (W,W Z)$ in the chart $(\mathbb{R}^2,(W,Z))$, we obtain a function with only normal crossings 
$$f \circ \sigma(W,Z) = W^4 Z^2 (Z - 1)(Z + 1),$$
with $\sigma := \sigma_1 \circ \sigma_2$. It is in particular invariant under the action of $G$ on the chart, given by $(W,Z) \mapsto (-W,-Z)$.
    
The fiber $f \circ \sigma^{-1}(0)$ have four irreducible components, given in the chart $(\mathbb{R}^2,(W,Z))$ by 
$$E_1 = \{Z = 0\},~E_2 = \{W = 0\},~E_3 = \{ Z = 1 \},~E_4 = \{ Z = - 1 \}.$$ 
The exceptional divisors $E_1$ and $E_2$, both isomorphic to a circle, intersect at a $G$-fixed point and $E_2$ intersects the irreducible components $E_3$ and $E_4$ of the strict transform of $f$ at two points exchanged by the action. 

Therefore, using the equivariant Denef-Loeser formula, since $N_1 = 2$, $\nu_1 = 1$, $N_2 = 4$, $\nu_2 = 2$ and $N_3 = N_4 = 1$, $\nu_3 = \nu_4 = 0$, we have after computation (use Example \ref{exevps})
\begin{eqnarray*}
Z^G_f(u,T) = \frac{T^2}{1-u^{-2} T^2} + (u^2 - u + 1) \frac{u^{-3} T^4}{1-u^{-3} T^4} & + &(u-1)\frac{u^{-1} T^2}{1-u^{-2} T^2} \frac{u^{-3} T^4}{1-u^{-3} T^4} \\
& + &(u-1)^2 \frac{u^{-3} T^4}{1-u^{-3} T^4} \frac{u^{-1} T^1}{1-u^{-1} T^1}.
\end{eqnarray*}
\\

In order to compute the naive equivariant zeta function of $h$, we consider the equivariant blowings-up $\sigma_1$ and $\sigma_2$ in the respective charts $(\mathbb{R}^2,(U,T))$ and $(\mathbb{R}^2,(R,S))$, where they are given by $\sigma_1(U,T) = (U,UT)$ and $\sigma_2(R,S) = (RS,S)$. The actions of $G$ on these charts are given by $(U,T) \mapsto (-U,-T)$ and $(R,S) \mapsto (R,-S)$, and we have
$$h \circ \sigma(R,S) = S^4 R^2 (R - 1) (R + 1)$$ 
in the chart $(\mathbb{R}^2,(R,S))$.

The four irreducible components of $(h \circ \sigma)^{-1}(0)$ are given by 
$$E_1' = \{R = 0\},~E_2' = \{S = 0\},~E_3' = \{ R = 1 \},~E_4' = \{ R = - 1 \}$$
and the exceptional divisor $E_2'$ intersects the strict transform of $h$ at two points that are both fixed by the action of $G$. Thus, 
\begin{eqnarray*}
Z^G_h(u,T) =  \frac{T^2}{1-u^{-2} T^2} + (u^2 - 2 u) \frac{u^{-3} T^4}{1 - u^{-3} T^4} & + & (u-1) \frac{u^{-1} T^2}{1-u^{-2} T^2} \frac{u^{-3} T^4}{1-u^{-3} T^4} \\
& + &(2 u(u-1)) \frac{u^{-3} T^4}{1-u^{-3} T^4} \frac{u^{-1} T^1}{1-u^{-1} T^1}. 
\end{eqnarray*}
 
In particular, 
$$Z^G_f(u,T) \neq  Z^G_h(u,T)$$
and therefore, with respect to the considered $\mathbb{Z}/2\mathbb{Z}$-action on $\mathbb{R}^2$, $f$ and $h$ are not $G$-blow-Nash equivalent via an equivariant blow-Nash isomorphism, by Theorem \ref{equivzetafuncinv}.
\end{ex}

\section{Examples} \label{examples}

In this section, we compute the equivariant zeta functions of several invariant Nash germs under linear actions of $G := \mathbb{Z}/2\mathbb{Z}$, using Denef-Loeser formula (Propositions \ref{equivDLform} and \ref{equivDLformsigns}). 

\begin{ex}[see also Example 3.6 of \cite{GF-MI}] \label{simpleexequivzetafuncDLform} Equip the affine plane $\mathbb{R}^2$ with any involution of the type $s : (x,y) \mapsto (\epsilon x , \epsilon' y)$ with $\epsilon, \epsilon' \in \{-1, 1\}$ 

\begin{itemize}
	\item Consider the invariant Nash germ $f : (x,y) \mapsto x^2 + y^2$ at $(0,0)$. The equivariant blowing-up of the plane at the origin gives an equivariant resolution $\sigma$ of the singularities of $f$ and the fiber $(f \circ \sigma)^{-1}(0)$ consists just in the exceptional divisor $E_1$ of the blowing-up equipped with the induced non-free action of $G$, and we obtain (see Example \ref{exevps})
$$Z_f^G(u,T) = (u-1) \left(u + 2 \frac{u}{u-1}\right) \frac{u^{-2} T^2}{1-u^{-2} T^2} = (u^2 + u)\frac{u^{-2} T^2}{1-u^{-2} T^2}.$$

Now, since $f$ is a positive function, we know that $Z_f^{G,-}(u,T) = 0$, and, since $\widetilde{E^{0,+}_{\{1\}}}$ is the boundary of a M\"obius band equipped with a non-free action of $G$,
$$Z_f^{G,+}(u,T) = \left(u + 2 \frac{u}{u-1}\right) \frac{u^{-2} T^2}{1-u^{-2} T^2}.$$

\item Consider the invariant Nash germ $h : (x,y) \mapsto - x^2 - y^4$. Two successive equivariant blowings-up provide an equivariant resolution of singularities $\tau$ of $h$. The two exceptional divisors $E_1'$ and $E_2'$, intersecting at one $G$-fixed point, constitute the fiber $(f \circ \tau)^{-1}(0)$ and we have
$$Z_f^G(u,T) = u^2 \frac{u^{-2} T^2}{1-u^{-2} T^2} + u^2 \frac{u^{-3} T^4}{1-u^{-3} T^4} + (u-1) u \frac{u^{-2} T^2}{1-u^{-2} T^2}\frac{u^{-3} T^4}{1-u^{-3} T^4}.$$

The sets $\widetilde{E^{0,+}_{\{1\}}}$ and $\widetilde{E^{0,+}_{\{2\}}}$ are both the boundary of a M\"obius band minus two points fixed under the induced action of $G$. Consequently,
$$Z_f^{G,-}(u,T) = u \frac{u^{-2} T^2}{1-u^{-2} T^2} + u \frac{u^{-3} T^4}{1-u^{-3} T^4} + 2 u \frac{u^{-2} T^2}{1-u^{-2} T^2}\frac{u^{-3} T^4}{1-u^{-3} T^4}$$
($Z_h^{G,+}(u,T) = 0$ since $h$ is negative).
\end{itemize}
\end{ex}

In Example \ref{exsingbound} below, the affine plane $\mathbb{R}^2$ is equipped with the action of $G$ given by the involution $s : (x,y) \mapsto (-x,y)$. We compute the naive equivariant zeta functions of the invariant Nash germs $f$, $g_k$, $k \geq 2$, and $h_k$, $k \geq 2$, at the origin of $\mathbb{R}^2$ given by
$$f(x,y) = \pm x^4 + y^3~~;~~g_k(x,y) = \pm x^{2k} \pm y^2~~;~~h_k(x,y) = x^2 y \pm y^k.$$
These germs are induced from the normal forms of the simple boundary singularities of manifolds with boundary, by unfolding the positive abscissa half-plane along the ordinate axis : see \cite{AGZV}. We will study the classification of the simple boundary singularities of Nash manifolds with boundary up to (equivariant) blow-Nash equivalence in a subsequent work.

\begin{ex} \label{exsingbound}

\begin{enumerate}
\item We begin with $f$. Consider the equivariant blowing-up $\sigma_1$ at the origin given in the chart $(\mathbb{R}^2, (X_1, Y_1))$ by $\sigma_1(X_1, Y_1) = (X_1, X_1 Y_1)$. The action of $G$ on the blowing-up is given in this chart by $s_1 : (X_1, Y_1) \mapsto (-X_1, -Y_1)$ and we have
$$f \circ \sigma_1 (X_1, Y_1) = X_1^3 (Y_1^3 \pm X_1).$$

We do three more successive blowings-up $\sigma_2$, $\sigma_3$ and $\sigma_4$ each given in the chart $(\mathbb{R}^2, (X_i, Y_i))$ by $\sigma_i(X_i, Y_i) = (X_i Y_i, X_i)$. The action of $G$ on the last blowing-up is given in the chart $(\mathbb{R}^2, (X_4, Y_4))$ by $s_4 : (X_4, Y_4) \mapsto (X_4, -Y_4)$ and we have
$$f \circ \sigma (X_4, Y_4) = X_4^3 Y_4^{12} (1 \pm X_4)$$
where $\sigma := \sigma_1 \circ \cdots \circ \sigma_4$.

The (equivariant) resolution tree of $f$ is the following, where $E_i(N_i, \nu_i)$ denotes the exceptional divisor of the blowing-up $\sigma_i$ with $N_i = mult_{E_i}~f \circ \sigma_i = mult_{E_i}~f \circ \sigma$ and $\nu_i = 1 + mult_{E_i}~jac~\sigma_i = 1 + mult_{E_i}~jac~\sigma$ :

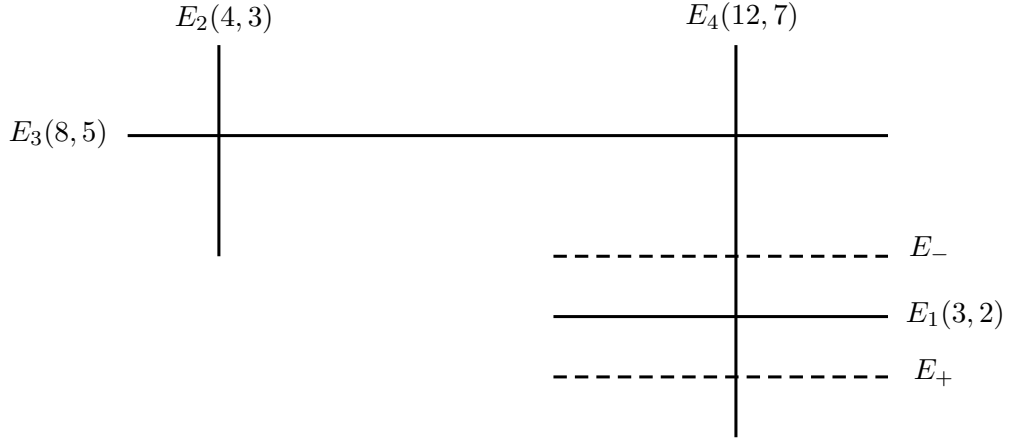
\psscalebox{1.0 1.0}
{
\begin{pspicture}(0,-3.2100153)(13.8,3.2100153)
\psline[linecolor=black, linewidth=0.04](3.6,2.0100152)(3.6,-0.78998476)
\psline[linecolor=black, linewidth=0.04](2.4,0.81001526)(12.4,0.81001526)(12.4,0.81001526)
\psline[linecolor=black, linewidth=0.04](10.4,2.0100152)(10.4,-3.1899848)(10.4,-3.1899848)
\psline[linecolor=black, linewidth=0.04](8.0,-1.5899848)(12.4,-1.5899848)(12.4,-1.5899848)
\pscustom[linecolor=black, linewidth=0.04]
{
\newpath
\moveto(3.2,2.8100152)
}
\rput[bl](3.02,2.1700153){$E_2(4,3)$}
\rput[bl](9.74,2.1900153){$E_4(12,7)$}
\rput[bl](0.84,0.65001523){$E_3(8,5)$}
\rput[bl](12.64,-1.7499847){$E_1(3,2)$}
\psline[linestyle=dashed, linecolor=black, linewidth=0.04](8.0,-0.78998476)(12.4,-0.78998476)(12.4,-0.78998476)
\psline[linestyle=dashed, linecolor=black, linewidth=0.04](8.0,-2.3899848)(12.4,-2.3899848)
\rput[bl](12.68,-0.86998475){$E_-$}
\rput[bl](12.74,-2.5099847){$E_+$}
\end{pspicture}
}

The sets $E_-$ and $E_+$ are the respective strict transforms of $f_-(x,y) = -x^3 + y^4$ and $f_+(x,y) = x^3 + y^4$. In both cases, the action of $G$ globally stabilizes the strict transform, the exceptional divisors and the intersections. Then Denef-Loeser formula provides the naive equivariant zeta function of $f$ (using also Example \ref{exevps}) :
\begin{eqnarray*}
Z_f^G(u,T) & = & \sum_{i = 1}^3 \frac{u^{-\nu_i+2} T^{N_i}}{1-u^{-\nu_i} T^{N_i}} + (u-1) \frac{u^{-6} T^{12}}{1-u^{-7} T^{12}} + (u-1) \frac{u^{-2} T^4}{1-u^{-3} T^4} \frac{u^{-5} T^8}{1-u^{-5} T^8} \\
& + & (u-1) \frac{u^{-6} T^{12}}{1-u^{-7} T^{12}} \left[ \frac{u^{-5} T^8}{1-u^{-5} T^8} + \frac{u^{-2} T^3}{1-u^{-2} T^3} + \frac{u^{-1} T^1}{1-u^{-1} T^1} \right]
\end{eqnarray*}

\item We now compute the naive equivariant zeta function of $g_k$ for $k \geq 3$. By $k$ successive equivariant blowings-up $\sigma_i$, $i = 1, \ldots, k$, given in charts $(\mathbb{R}^2, (X_i,Y_i))$ by $\sigma_i : (X_i, Y_i) \mapsto (X_i, X_i Y_i)$, we resolve the singularities of $g_k$ :
$$g_k \circ \sigma (X_k, Y_k) = X_k^{2k} ( \pm 1 \pm Y_k^2)$$ 
with $\sigma := \sigma_1 \circ \cdots \circ \sigma_k$.

First, we deal with the case $g_k(x,y) = \pm(x^{2k} - y^2)$. In this case, we have the following resolution tree for $g_k$ :

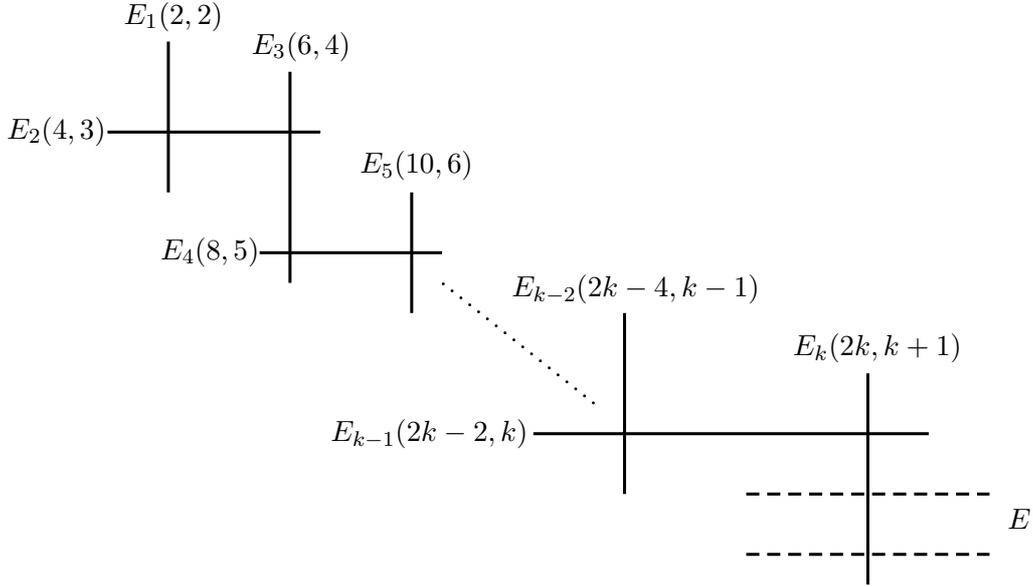
\psscalebox{1.0 1.0}
{
\begin{pspicture}(0,-3.845)(13.42,3.845)
\psline[linecolor=black, linewidth=0.04](1.32,2.155)(4.12,2.155)
\psline[linecolor=black, linewidth=0.04](2.12,3.355)(2.12,1.355)(2.12,1.355)
\psline[linecolor=black, linewidth=0.04](3.72,2.955)(3.72,0.155)
\psline[linecolor=black, linewidth=0.04](3.32,0.555)(5.72,0.555)(5.72,0.555)
\psline[linecolor=black, linewidth=0.04](5.32,1.355)(5.32,-0.245)(5.32,-0.245)
\psline[linecolor=black, linewidth=0.04](8.12,-0.245)(8.12,-2.645)(8.12,-2.645)
\psline[linecolor=black, linewidth=0.04](6.92,-1.845)(12.12,-1.845)(12.12,-1.845)
\psline[linecolor=black, linewidth=0.04](11.32,-1.045)(11.32,-3.845)
\psline[linestyle=dashed, linecolor=black, linewidth=0.04](9.72,-2.645)(12.92,-2.645)(12.92,-2.645)
\psline[linestyle=dashed, linecolor=black, linewidth=0.04](9.72,-3.445)(12.92,-3.445)(12.92,-3.445)
\rput[bl](0.0,1.975){$E_2(4,3)$}
\rput[bl](1.54,3.495){$E_1(2,2)$}
\rput[bl](3.22,3.095){$E_3(6,4)$}
\rput[bl](2.02,0.375){$E_4(8,5)$}
\rput[bl](4.64,1.515){$E_5(10,6)$}
\rput[bl](6.64,-0.125){$E_{k-2}(2k-4,k-1)$}
\rput[bl](4.26,-2.045){$E_{k-1}(2k-2,k)$}
\rput[bl](10.34,-0.905){$E_k(2k,k+1)$}
\rput[bl](13.16,-3.105){$E$}
\psline[linestyle=dotted, linecolor=black, linewidth=0.04](5.72,0.155)(7.72,-1.445)(7.72,-1.445)
\end{pspicture}
}

In the chart $(\mathbb{R}^2, (X_k, Y_k))$, the action of $G$ is given by
$$s_k : (X_k, Y_k) \mapsto \begin{cases} (-X_k, -Y_k) \mbox{ if $k$ is odd,} \\ (-X_k, Y_k) \mbox{ if $k$ is even.} \end{cases}$$
Thus, the intersection points of the strict transform $E$ of $g_k$ with the exceptional divisor $E_k$ are exchanged under the involution if $k$ is odd and fixed if $k$ is even.

Consequently, after computation we obtain 
\begin{eqnarray*}
Z_{g_k}^G(u,T) & = & \frac{T^{2}}{1-u^{-2} T^{2}} + (u-1) \left[ \sum_{j = 2}^{k-1} \frac{u^{-j} T^{2j}}{1-u^{-(j+1)} T^{2j}} + \sum_{j=1}^{k-1} \frac{u^{-j} T^{2j}}{1-u^{-(j+1)} T^{2j}} \frac{u^{-(j+2)} T^{2j+2}}{1-u^{-(j+2)} T^{2j+2}}\right]\\
& + & \Lambda_k(u,T)
\end{eqnarray*}
with
$$\Lambda_k(u,T) = \begin{cases} (u^2 - u +1) \frac{u^{-(k+1)} T^{2k}}{1-u^{-(k+1)} T^{2k}} + (u-1)^2 \frac{u^{-(k+1)} T^{2k}}{1-u^{-(k+1)} T^{2k}} \frac{u^{-1} T^1}{1-u^{-1} T^1} \mbox{ if $k$ is odd, }\\
(u-2) \frac{u^{-k} T^{2k}}{1-u^{-(k+1)} T^{2k}} + 2 u (u-1) \frac{u^{-(k+1)} T^{2k}}{1-u^{-(k+1)} T^{2k}} \frac{u^{-1} T^1}{1-u^{-1} T^1} \mbox{ if $k$ is even.} 
\end{cases}
$$

In the case $g_k(x,y) = \pm (x^{2k} + y^2)$, there is no strict transform and the naive equivariant zeta function of $g_k$ is given by the same formula as above with
$$\Lambda_k(u,T) = \frac{u^{-k+1} T^{2k}}{1-u^{-(k+1)} T^{2k}}.$$

\item Let us next consider the invariant Nash germ $h_k$, $k \geq 3$. We first look at the case $k$ odd. If $k = 2p+1$ with $p \in \mathbb{N}$, then $p$ successive equivariant blowings-up $\sigma_i : (X_i, Y_i) \rightarrow (X_i Y_i, Y_i)$ provide the function with only normal crossings
$$h_k \circ \sigma_1 \circ \cdots \circ \sigma_p(X_p, Y_p) = Y_p^k(X_p^2 \pm 1),$$   
together with the following resolution tree

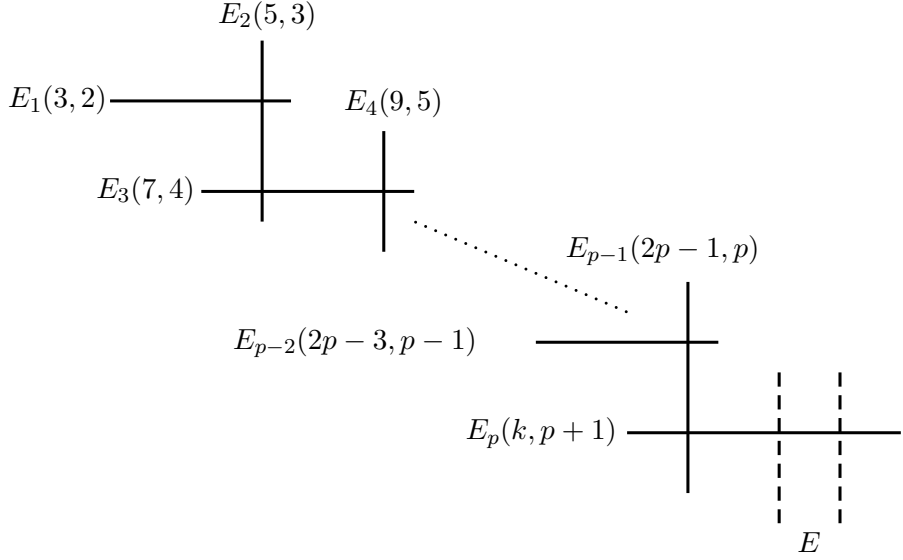
\psscalebox{1.0 1.0}
{
\begin{pspicture}(0,-3.635)(11.74,3.635)
\psline[linecolor=black, linewidth=0.04](1.34,2.345)(3.72,2.345)
\psline[linecolor=black, linewidth=0.04](3.34,3.145)(3.34,0.745)(3.34,0.745)
\psline[linecolor=black, linewidth=0.04](2.54,1.145)(5.34,1.145)
\psline[linecolor=black, linewidth=0.04](4.94,1.945)(4.94,0.345)
\psline[linecolor=black, linewidth=0.04](6.94,-0.855)(9.34,-0.855)
\psline[linecolor=black, linewidth=0.04](8.94,-0.055)(8.94,-2.855)
\psline[linecolor=black, linewidth=0.04](8.14,-2.055)(11.74,-2.055)
\psline[linestyle=dashed, linecolor=black, linewidth=0.04](10.94,-1.255)(10.94,-3.255)
\rput[bl](0.0,2.165){$E_1(3,2)$}
\rput[bl](2.76,3.285){$E_2(5,3)$}
\rput[bl](1.16,0.965){$E_3(7,4)$}
\rput[bl](4.42,2.125){$E_4(9,5)$}
\rput[bl](2.96,-1.055){$E_{p-2}(2p-3,p-1)$}
\rput[bl](7.34,0.165){$E_{p-1}(2p-1,p)$}
\rput[bl](6.02,-2.255){$E_p(k, p+1)$}
\rput[bl](10.38,-3.635){$E$}
\psline[linestyle=dashed, linecolor=black, linewidth=0.04](10.14,-1.255)(10.14,-3.255)
\psline[linestyle=dotted, linecolor=black, linewidth=0.04](5.34,0.745)(8.14,-0.455)
\end{pspicture}
}

(the above resolution tree corresponds to the case $h_k(x,y) = x^2 y - y^k$ ; in the case $h_k(x,y) = x^2 y + y^k$, there is no strict stransform). In the chart $(\mathbb{R}^2, (X_p,Y_p))$, the action of $G$ is given by $s_p : (X_p, Y_p) \mapsto (-X_p, Y_p)$, hence exchanges the intersection points of the strict transform $E$ of $h_k$ with the exceptional divisor $E_p$. 

If now we suppose $k = 2p$ with $p \in \mathbb{N}\setminus \{0,1\}$, by doing the same first $p-1$ successive equivariant blowings-up $\sigma_1, \ldots, \sigma_{p-1}$ as above, regarded in the same charts, we obtain
$$h_k \circ \sigma_1 \circ \cdots \circ \sigma_{p-1}(X_{p-1}, Y_{p-1}) = Y_{p-1}^{2p-1} (X_{p-1}^2 \pm Y_{p-1}),$$

We obtain the equivariant resolution of singularities of $h_k$ by two more equivariant blowings-up, getting the following tree :
\\
\psscalebox{1.0 1.0}
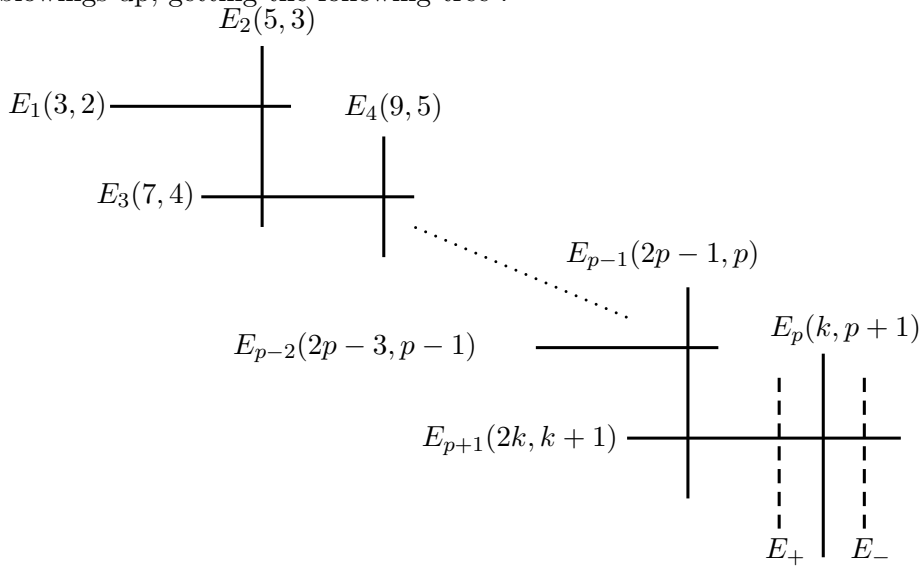
{
\begin{pspicture}(0,-3.685)(11.8,3.685)
\psline[linecolor=black, linewidth=0.04](1.34,2.395)(3.72,2.395)
\psline[linecolor=black, linewidth=0.04](3.34,3.195)(3.34,0.795)(3.34,0.795)
\psline[linecolor=black, linewidth=0.04](2.54,1.195)(5.34,1.195)
\psline[linecolor=black, linewidth=0.04](4.94,1.995)(4.94,0.395)
\psline[linecolor=black, linewidth=0.04](6.94,-0.805)(9.34,-0.805)
\psline[linecolor=black, linewidth=0.04](8.94,-0.005)(8.94,-2.805)
\psline[linecolor=black, linewidth=0.04](8.14,-2.005)(11.74,-2.005)
\psline[linestyle=dashed, linecolor=black, linewidth=0.04](11.26,-1.205)(11.26,-3.205)
%\psline[linestyle=dashed, linecolor=black, linewidth=0.04](11.34,-1.205)(11.34,-3.205)
\rput[bl](0.0,2.215){$E_1(3,2)$}
\rput[bl](2.76,3.335){$E_2(5,3)$}
\rput[bl](1.16,1.015){$E_3(7,4)$}
\rput[bl](4.42,2.175){$E_4(9,5)$}
\rput[bl](2.96,-1.005){$E_{p-2}(2p-3,p-1)$}
\rput[bl](7.34,0.215){$E_{p-1}(2p-1,p)$}
\rput[bl](5.44,-2.225){$E_{p+1}(2k, k+1)$}
\rput[bl](9.96,-3.685){$E_+$}
\psline[linestyle=dashed, linecolor=black, linewidth=0.04](10.14,-1.205)(10.14,-3.205)
\psline[linestyle=dotted, linecolor=black, linewidth=0.04](5.34,0.795)(8.14,-0.405)
\psline[linecolor=black, linewidth=0.04](10.72,-0.885)(10.72,-3.585)
\rput[bl](10.02,-0.765){$E_p(k,p+1)$}
\rput[bl](11.08,-3.685){$E_-$}
\end{pspicture}
}

(where $E_-$ and $E_+$ are the respective strict transforms of $x^2 y - y^k$ and $x^2 y + y^k$), all the intersection points being fixed by the action of $G$.

\item Finally, we take a look at the germs $g_2$ and $h_2$. By applying the same two equivariant blowings-up regarded in the same charts, we obtain isomorphic resolution trees for $h_2$ and $g_2 : (x,y) \mapsto \pm (x^4 - y^2)$ with same multiplicities, the action of $G$ fixing all intersection points. Notice that the case $g_2(x,y) = \pm (x^4 + y^2)$ is treated in Example \ref{simpleexequivzetafuncDLform} up to equivariant Nash equivalence.

\end{enumerate}

\end{ex}

 \vspace{0.5cm}
Fabien PRIZIAC
\\
Department of Mathematics
\\
Faculty of Science
\\
Saitama University
\\
255 Shimo-Okubo, Sakura-ku, Saitama City
\\
338-8570
\\
JAPAN
\\
priziac.fabien@gmail.com

\end{document}